\documentclass[9pt,a4paper]{article}

\usepackage{amsmath,amssymb}
\usepackage{graphicx}
\usepackage{ulem}
\usepackage{lipsum}
\usepackage{amsfonts}
\usepackage{graphicx}
\usepackage{authblk}
\usepackage{epstopdf}
\usepackage{algorithmic}
\ifpdf
  \DeclareGraphicsExtensions{.eps,.pdf,.png,.jpg}
\else
  \DeclareGraphicsExtensions{.eps}
\fi
\usepackage{xcolor,hyperref}
\definecolor{darkblue}{rgb}{0.0,0.0,0.5}
\definecolor{darkgrey}{rgb}{0.5,0.1,0.1}
\hypersetup{colorlinks,breaklinks,
            linkcolor=darkblue,urlcolor=darkblue,
            anchorcolor=darkblue,citecolor=darkblue}

\usepackage{enumitem}
\setlist[enumerate]{leftmargin=.5in}
\setlist[itemize]{leftmargin=.5in}
\usepackage{todonotes} 

\usepackage{tikz}
\newcommand*\circled[1]{\tikz[baseline=(char.base)]{
            \node[shape=circle,draw,inner sep=2pt] (char) {#1};}}

\newtheorem {theorem} {Theorem}
\newtheorem {proposition} [theorem] {Proposition}

\newtheorem {definition} {Definition}
\newtheorem {remark} {Remark}

\newtheorem {proof} {Proof}

\newcommand{\ds}{\displaystyle}
\newcommand{\eps}{\varepsilon}
\newcommand{\beq}{\begin{equation}}
\newcommand{\eeq}{\end{equation}}

\begin{document}

\title{\textbf{Spike-adding and reset-induced canard cycles}\\ \textbf{in adaptive integrate and fire models}}

\author[1,2]{Mathieu Desroches}
\author[3]{Piotr Kowalczyk}
\author[4,5]{Serafim Rodrigues}
\affil[1]{MathNeuro Team, Inria Sophia Antipolis M{\'e}diterran{\'e}e,
            France, \href{mailto:mathieu.desroches@inria.fr}{mathieu.desroches@inria.fr}}
\affil[2]{Universit{\'e} C{\^o}te d'Azur, Nice, France}
\affil[3]{Department of Mathematics,
            Wroc{\l}aw University of Science and Technology,
            Poland, \href{piotr.s.kowalczyk@pwr.edu.pl}{piotr.s.kowalczyk@pwr.edu.pl}}
\affil[4]{MCEN Team, BCAM - Basque Center for Applied Mathematics,
            Bilbao, Spain, \href{srodrigues@bcamath.org}{srodrigues@bcamath.org}}
\affil[5]{Ikerbasque - The Basque Foundation for Science}
\date{}

\maketitle
\begin{abstract}
We study a class of planar integrate and fire (IF) models called \textit{adaptive integrate and fire (AIF) models}, which possesses an adaptation variable on top of membrane potential, and whose subthreshold dynamics is piecewise linear (PWL). These AIF models therefore have two reset conditions, which enable bursting dynamics to emerge for suitable parameter values. Such models can be thought of as hybrid dynamical systems. We consider a particular slow dynamics within AIF models and prove the existence of bursting cycles with $N$ resets, for any integer $N$. Furthermore, we study the transition between $N$- and $(N+1)$-reset cycles upon vanishingly small parameter variations and prove (for $N=2$) that such transitions are organised by canard cycles. Finally, using numerical continuation we compute branches of bursting cycles, including canard-explosive branches, in these AIF models, by suitably recasting the periodic problem as a two-point boundary-value problem.
\end{abstract}
%

%------------------------------------------------------------------------------------------
\section{Introduction}
\label{sec:intro}
%------------------------------------------------------------------------------------------
Dynamical systems with reset are not only interesting for themselves, as mathematical objects, they are also extremely relevant in applications. There are at least two application areas where these systems are pertinent: first, in neuroscience, as models of neural activity; second, in electronic engineering where resets are one means of control. Neuron models with reset belong to the general class of \textit{integrate and fire (IF) models}, which correspond to hybrid systems following a differential equation together with one or several reset (algebraic) rules applied as soon as one variable attains a pre-determined value referred to as \textit{threshold}.

In the neuronal context, the variable undergoing a threshold event is the membrane potential $v$ and the reset consists in stopping the integration of the model and restarting instantaneously with a new, lower, value for the membrane potential. This mimics the behaviour of neurons' membrane potential as it crosses the excitability threshold with the key difference that real neurons (and other neuron models) emit an action potential (or spike) as soon as they cross their firing threshold, which corresponds to a large oscillation in $v$ which naturally brings the potential down, close to its rest value. IF models do not compute any spikes but instead mark threshold crossings as fixed values of $v$ and reset $v$ instantaneously to value close to rest. In this way, only the subthreshold dynamics is effectively governed by an evolution equation, and the superthreshold dynamics is replaced by an instantaneous reset. Therefore, the reset allows such one-dimensional models to support limit cycles and hence the reset acts as a second state variable. 

IF neuron models can be classified according to their subthreshold dynamics, whether it be linear (LIF models) or nonlinear (e.g. quadratic IF or QIF). They are very popular in neuroscience, both because they are amenable to theoretical analysis and also because they are very much suitable to computations within large networks. What is more, IF models are compatible with electronic implementation as so-called artificial neurons, and a large body of current research is focused on implementing networks of neuro-inspired systems akin to IF models within the fast-developing area of \textit{neuromorphic computing}~\cite{chen2019,furber2016}.

A number of IF models (e.g. LIF, QIF) are one-dimensional, the membrane potential $v$ being the only state variable. However, one can build and study IF type models with additional state variables accounting to, e.g., gating dynamics, in which case the reset can then affect more than one variable. As a consequence, such models can reproduce more complex neuronal oscillations when classical one-dimensional IF models can only reproduce spiking activity. This is in particular the case of the \textit{Izhikevich model}~\cite{izhikevich2003}, which is effectively a QIF with one slower recovery variable and hence two reset conditions. The Izhikevich model can reproduce bursting activity. Indeed, by considering the reset as one additional variable, the resulting model can be seen as having two fast ($v$ and ``reset'') and one slow ($w$) variables, which makes it compatible with minimal slow-fast models of bursting activity; see Section~\ref{sec:systems} below.

In the present work, we study a family of IF models called \textit{adaptive integrate and fire (AIF)} models, with piecewise-linear (PWL) subthreshold dynamics and one slow adaptation variable. Hence, this model can be considered as a PWL version of Izhikevich's model and it can produce bursting oscillations. This family of IF models has been introduced in~\cite{karbowski2000} and further studied in~\cite{coombes2012}. We study such bursting oscillations using both theoretical and computational tools. In particular, we prove the existence of limit cycles with $N$-resets, for any $N$, under certain assumptions on the slow dynamics of these AIF models. We also put a strong emphasis on the transition between $N$- and $(N+1)$-reset cycles in the model, under very small parameter variations, motivated by direct simulations that suggest the occurrence of \textit{canard cycles}~\cite{BCDD1981,eckhaus1983,krupa01} in such slow-fast systems with reset organising \textit{spike-adding transitions}~\cite{desroches2013,desroches2016} that correspond, in the IF context, to reset-adding transitions.

Canards are special solutions of multiple-timescale system that possess strong and unexpected properties. In particular, they follow repelling (locally) invariant manifolds for long time intervals, and they exist in very narrow (exponentially-small) parameter ranges. Canards have been found in many models of neural activity, both biophysical (e.g. Hodgkin-Huxley, Morris-Lecar) and phenomenological (FitzHugh-Nagumo) and their role in organising the transition between different activity regimes has long been established: they provide good approximation of firing threshold in spiking models, and they also exist along branches of bursting cycles, precisely at the transition between cycles with $N$- and $(N+1)$-spikes per burst. However, they have been little studied in the context of IF models and this is the main topic of the present work. 

We prove the existence of canard cycles at the interface, in parameter space, between $2$- and $3$-reset cycles, as an examplary spike-adding transition. On the computational side, we use numerical continuation to compute families of such cycles, including canard-explosive branches during spike-adding transitions, by recasting the planar periodic problem with PWL dynamics and resets into a two-point boundary value problem (BVP) of an extended problem formed by considering multiple copies of the original system. This BVP is suitable for standard numerical continuation, which we perform with the software package \textsc{auto}~\cite{auto}, showing in passing that one can compute parametrised families of limit cycles in IF systems within this standard framework at little cost. Finally, we show how the bursting resulting from this family of AIF models is akin to \textit{square-wave bursting} and highlight the similarities and differences with smooth models of square-wave bursting such as the Hindmarsh-Rose burster~\cite{desroches2013,hindmarsh1984,shilnikov2008}.

The rest of the paper is organised as follows. In Section~\ref{sec:systems}, we first introduce the class of systems that we will study, namely planar slow-fast systems with resets akin to integrate-and-fire neuron models with a slow adaptation variable. Then we present the dynamics of the two limiting problems obtained in the singular limit, that is, the slow and fast subsystems, respectively. Section~\ref{sec:dynamics} is devoted to explaining the spike-adding transitions observed in our model example. We provide a theoretical proof for the existence of cycles with resets as well as the canard-induced transition that connect them in parameter space. To fully unveil the canard-mediated spike-adding transition of interest, we resort to using numerical continuation. We explain in Section~\ref{sec:nbif} how to do continue cycles with reset across such transitions 
using two-point boundary-value problems. Section~\ref{sec:cplm} gives a brief description of the dynamics of a modified AIF model where the slow adaptation dynamics is coupled to the voltage variable. Then, Section~\ref{sec:snsmc} provides a comparison between AIF models and smooth square-wave bursters. Finally, we summarise our finding in Section~\ref{sec:conc} and highlight interesting directions for future work.  

%------------------------------------------------------------------------------------------
\section{AIF models and their slow-fast analysis}
\label{sec:systems}
%------------------------------------------------------------------------------------------
%---
\subsection{Adaptive integrate-and-fire (AIF) models}
%---
We consider the following family of AIF models
\begin{equation}\label{eq:modelgen}
\begin{split}
v' &= |v|-w+I\\
w' &=\eps F(v,w),
\end{split}
\end{equation}
where $I$ is a regular parameter representing an applied constant current, $0<\eps\ll1$ is a small parameter, the $0$ limit of which is relevant to the full system, and $F$ is a smooth function, which we will take to be linear in both arguments, for simplicity and without loss of generality on the dynamics of interest. We append to system~\eqref{eq:modelgen} the following reset rules: for $v=v_{\mathrm{thr}}$, we have  
\begin{equation}\label{eq:reset}
(v, \;w) \longrightarrow \left(v_{\mathrm{res}},\; w+k\right),
\end{equation}
meaning that as soon as $v$ reaches some threshold value $v_{\mathrm{thr}}$, the values of state variables $v$ and $w$ are reset to $v_{\mathrm{res}}$ and $w+k$, respectively. 
System~\eqref{eq:modelgen} is parametrised by the fast time $\tau$ and it is useful to rescale time so as to introduce the slow time $t=\eps\tau$, which brings system~\eqref{eq:modelgen} in its slow-time formulation, namely
\begin{equation}\label{eq:modelgenbis}
\begin{split}
\eps\dot{v} &= |v|-w+I\\
~~\dot{w}  &= F(v,w).
\end{split}
\end{equation}
The reset rules~\eqref{eq:reset} remain unchanged in the slow-time formulation. 

Classical IF models like the linear one (LIF) or the quadratic one (QIF) possess only one state variable, the membrane potential, and one reset rule associated with it. However, these minimal IF models can only account for rest and spiking dynamics. Emulating more complex dynamical regimes, such as that of \textit{bursting oscillations}, requires to introduce a second variable, which usually accounts for recovery or adaptation of the membrane dynamics and allows for a modulations of rest and spiking states in time; the second variable will also be subject to reset conditions. 

In the present work, we will study system~\eqref{eq:modelgen} in several configurations, depending on the slow adaptation dynamics given by function $F$, in the regime where such systems can produce bursting oscillations, that it, oscillatory solutions that alternate between a fast phase termed \textit{burst} and corresponding to spiking with a slow amplitude modulation, and a slow phase termed \textit{quiescence} corresponding to quasi-stationary dynamics. An example of such a periodic bursting solution obtained for $F(v,w)=-w$ is depicted in Figure~\ref{fig:dynaif}; its burst phase has 5 spikes, represented by 5 threshold-crossings and resets. Panel (a) displays the solution in the phase plane $(v,w)$ together the critical manifold $S_0$ (see below for a definition), the switching line $\{v=0\}$, the reset line $\{v=v_{\mathrm{thr}}\}$, the threshold line $\{v=v_{\mathrm{res}}\}$. Panel (b) shows the $v$ time profile for this solution together with the reset and the threshold lines. This time series has several key features of classical \textit{square-wave bursting}~\cite{izhikevich2000} and we will see that system~\eqref{eq:modelgen} can indeed be seen as a piecewise-linear square-wave burster with reset.

\begin{figure}[!t]
\centering
\includegraphics[width=\textwidth]{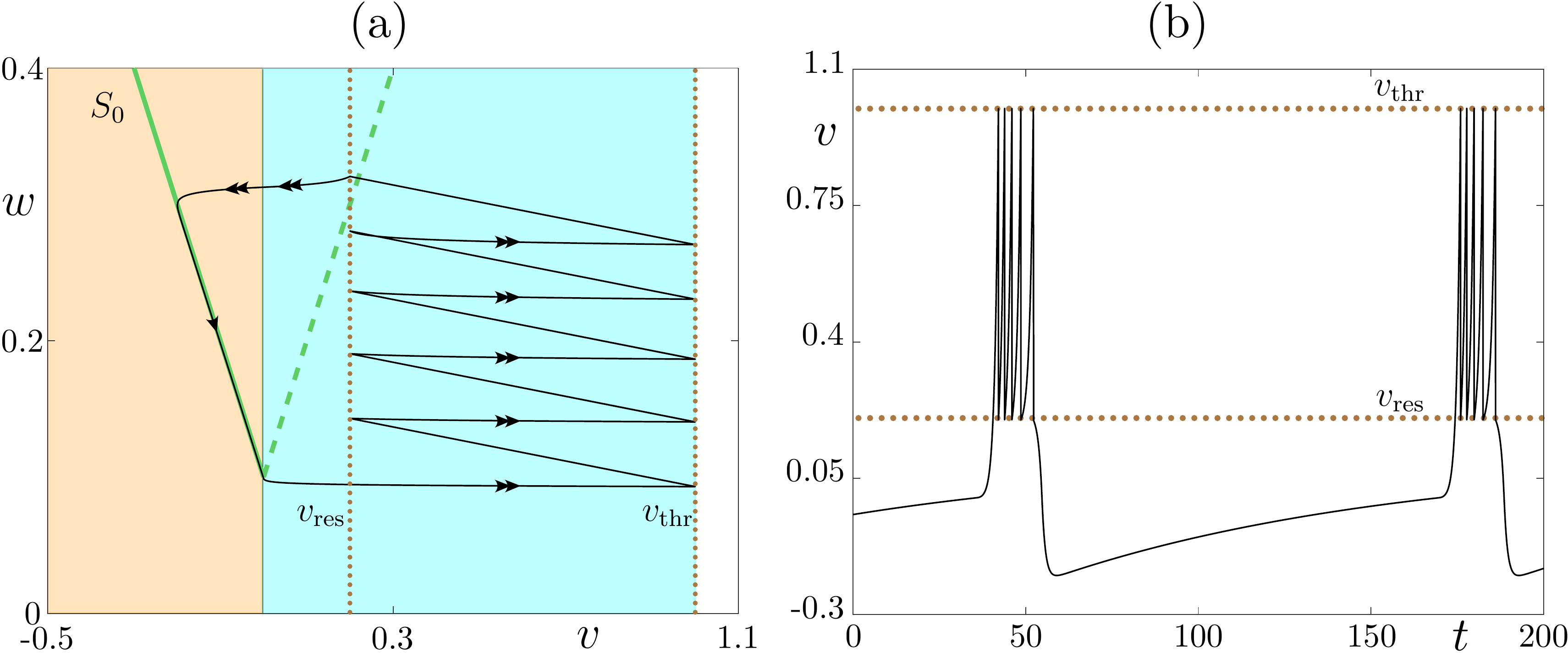}
\caption{Typical bursting dynamics of system~\eqref{eq:modelgen} for $F(u,w) = -w$ and the following fixed parameter values: $I=0.1$, $\eps=0.01$, $b=0$, $v_{\mathrm{res}}=0.2$,  $v_{\mathrm{thr}}=1$, $k=0.05$. In (a), the stable five-reset solution is shown together with the critical manifold $S_0$, the reset line $\{v=v_{\mathrm{res}}\}$ and the threshold line $\{v=v_{\mathrm{thr}}\}$. In (b), the time profile of the $v$-component of the solution is shown, together with reset and threshold lines.}
\label{fig:dynaif}
\end{figure}
The presence of the small parameter $\eps$ in system~\eqref{eq:modelgen} endows it with a slow-fast structure, whereby $v$ is a fast variable and $w$ is a slow variable. In the form given in~\eqref{eq:modelgen}, it is parametrised by the fast time $\tau$ and~\eqref{eq:modelgen} is the fast-time version of the model under consideration. Its $\eps=0$ limit yields the \textit{fast subsystem} or \textit{layer problem}, that is, a family of one-dimensional ODEs on $v$ parametrised by $w$, which becomes a parameter in this singular limit, together with one reset condition. 
In contrast, taking the $\eps = 0$ limit of the AIF model in its slow-time parametrisation~\eqref{eq:modelgenbis} yields a differential-algebraic equation (DAE) termed \textit{slow subsystem} or \textit{reduced system}. It consists of a differential equation for $w$ (unchanged from~\eqref{eq:modelgenbis}) constrained by the algebraic condition $0=|v|-w+I$, without any reset conditions.

The \textit{critical manifold} $S_0$ of system~\eqref{eq:modelgen}--\eqref{eq:reset}, or equivalently of system~\eqref{eq:modelgenbis}--\eqref{eq:reset}, is by definition the $\eps=0$ limit of its fast nullcline. Therefore in the present case we have $$S_0=\left\{0=|v|-w+I\right\}.$$ Note that both the slow-time and fast-time parametrisations of the AIF model are equivalent as long as $\eps\neq 0$. However their $\eps=0$ limits differ, and this is why the system is said to be singularly perturbed. The critical manifold plays a key role in both subsystems: it is the one-dimensional phase space of the slow subsystem as well as the set of equilibria of the fast subsystem.

%---
\subsection{Fast subsystem analysis of~\eqref{eq:modelgen}--\eqref{eq:reset}}
\label{sec:fastsub}
%---
%
%
\begin{figure}[!t]
\centering
\includegraphics[width=\textwidth]{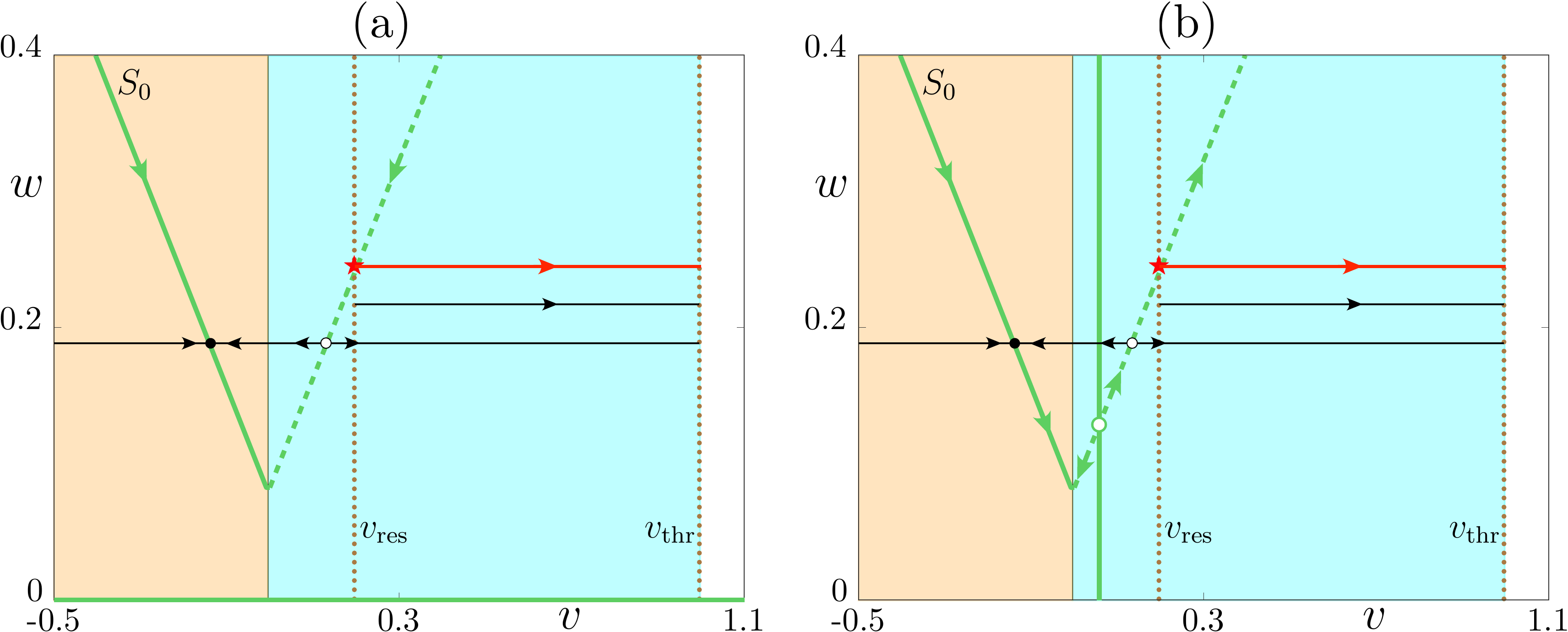}
\caption{Singular flows of~\eqref{eq:modelgen}--\eqref{eq:reset} we consider from the family of adaptive integrate-and-fire systems: slow flow (green arrows) on the critical manifold $S_0$; fast flow (black arrows) along fast fibers (black lines) with stable (filed circles) and unstable (unfiled circles); fast subsystems limit cycles (black segment with arrow) and nonsmooth homoclinic bifurcation (red segment with arrow and star). In panel (a), we take $F(v,w)=-w$; in panel (b), we consider $F(v,w)=v-b$, which allows for a slow flow equilibrium (unfiled green circle).}
\label{fig:singflows}
\end{figure}
The fast subsystem does not depend on the choice of $F$ since the dynamics of $w$ is frozen in the $\eps=0$ limit of~\eqref{eq:modelgen}. It reads
\begin{equation}\label{eq:fastsubsystem}
\begin{split}
v' &= |v|-w+I\\
w' &= 0 \\
v &=v_{\mathrm{thr}} \longrightarrow v=v_{\mathrm{res}}.
\end{split}
\end{equation}
Given that the equilibria of the fast subsystem must lie on $S_0$, one can easily show that~\eqref{eq:fastsubsystem} has no equilibria for $w<I$, one equilibrium for $w=I$ --- which can be considered stable for the $\{v\leq0\}$ system and unstable for the $\{v\geq0\}$ system --- and two equilibria for $w>I$, namely a stable one located in the $\{v\leq0\}$ zone and an unstable one located in the $\{v\geq0\}$ zone.
Hence, the system undergoes a nonsmooth saddle-node bifurcation of equilibria at $w=I$.

Due to the reset condition, one can easily prove that the fast subsystem possesses one stable limit cycle for $w < v_{\mathrm{res}}+I$, that is, up to the intersection point between the reset line and the critical manifold. As $w$ approaches this value from below, the period of the cycle grows unboundedly and when $w=v_{\mathrm{res}}+I$, then the first reset brings the system exactly at the (unstable) equilibrium where it stays for all future times. Therefore, what happens at $w=v_{\mathrm{res}}+I$ is a nonsmooth homoclinic bifurcation; see Figure~\ref{fig:singflows}(a) where the fast flow is indicated on lines $y=\mathrm{constant}$ with black arrows, stable (unstable) fast subsystem equilibria represented by a black dot (circle) and the nonsmooth homoclinic connection by a red segment.

\subsection{Slow subsystem analysis of~\eqref{eq:modelgen}--\eqref{eq:reset}}
\label{sec:slowsub}
Contrary to the fast subsystem, the slow subsystem does depend upon the choice of $F$. We will consider two cases with different slow flow directions in each one of them and hence different spike-adding sequences albeit organised by canard solutions in both cases. The first case that we will study, and which our theoretical results will be based upon, corresponds to the case $F(v,w)=b-w$, where $b$ is a parameter. Hence the slow subsystem, $\eps=0$ limit of system~\eqref{eq:modelgenbis}, becomes
\begin{equation}\label{eq:model1}
\begin{split}
0 &= |v|-w+I\\
~~\dot{w}  &= b - w.
\end{split}
\end{equation}

For $b=0$, the flow of the slow subsystem, typically referred to as \textit{slow flow}, is trivial and points down along each branch of the critical manifold without having any equilibria; see Figure~\ref{fig:singflows}(a) where the slow flow is indicated by green arrows on the critical manifold $S_0$. In particular, the corner point of the critical manifold is the PWL equivalent of a \textit{jump point}. This means that, for $\eps>0$ small enough, trajectories switch from slow to fast motion near this point. With a two-piece PWL critical manifold, this point cannot be a \textit{canard point} (i.e. a turning point of the slow flow) as, in order to have a canard point in a PWL setting, one must have three linearity zones to approximate the quadratic fold point of the critical manifold of the smooth slow-fast system~\cite{arima1997,desroches2018}. For $b = I$, there is a semistable equilibrium point on $S_0$, and for $b > I$ there is a pair of stable equilibria on $S_0$ which therefore appear via a nonsmooth saddle-node bifurcation at $b=I$. However, we will not consider the case $b \not = 0$ here.  

In Section~\ref{sec:cplm}, we will also consider an alternative AIF model corresponding to a type II neuronal model, for which we take $F(v,w)=v-b$. In this case, the slow flow has an equilibrium that is unstable for $b>0$, and this allows for the slow flow to go up along the repelling part of the critical manifold, see Figure~\ref{fig:singflows}(b). This will have some importance when investigating the reset-induced canard solutions in this configuration.

%----------------------------------------------------------------------------------------------------------------------
\section{Spike-adding in system~\eqref{eq:modelgen}--\eqref{eq:reset} with $F(v,w)=b-w$}
\label{sec:dynamics}
%----------------------------------------------------------------------------------------------------------------------
%---
\subsection{Setup}
%---
The present work is mainly focused on explaining so-called spike-adding occurring under parameter variation in system~\eqref{eq:modelgen} in the case where $F(v,w)=b-w$, for which we will provide a theoretical analysis of the phenomenon using methods from nonsmooth dynamical systems theory~\cite{bernardo2008}. We will give elements concerning the case $F(v,w)=v-b$ at the end of the paper mostly to highlight the similarities and differences in the spike-adding transitions. To do so, we will first define limit cycles with $n$ resets, for some integer $n\geq 1$, which we will term \textit{$n$-reset cycles}. Let us also denote by $\phi$ an evolution operator for $t \geq 0$ of system~\eqref{eq:modelgen} (or equivalently~\eqref{eq:modelgenbis}) with reset condition~\eqref{eq:reset}, which is a forward time composition of flows and resets. Operator $\phi$ satisfies the standard two properties of evolution operators, that is $\phi^0 = \mbox{id}$, where $\mbox{id}$ denotes the identity, and $\phi^{t + s}  =   \phi^{t}\circ \phi^{s}$, where times $t$ and $s$ are both positive and ``$\circ$'' denotes composition. Note however that the evolution operator $\phi$ includes resets, which act in zero time, and their order of occurrence is strictly determined given an initial state $(v_0,w_0)$.  
\begin{definition}\label{def:nrespercyc}
 An \textit{$n$-reset periodic cycle} $L(v,w,n)$, where $n$ is a positive integer, is a subset of the state space of system \eqref{eq:modelgen}(or \eqref{eq:modelgenbis})--\eqref{eq:reset} such that each point $P_0  = (v_0,w_0)\in L$ satisfies $\phi(v_0,w_0, t) = \phi(v_0,w_0, t + T)\in L$, for some $T > 0$ and any $t\geq 0$. Moreover, there are  $n$ time instants when system \eqref{eq:modelgen}(or \eqref{eq:modelgenbis})--\eqref{eq:reset} reaches the threshold value $v  = v_{\mathrm{thr}}$ for $t\in (0, \, T]$, starting an evolution at any point $(v_0,w_0)\in L$.
\end{definition}
 
 \begin{definition}\label{def:nrescyc}
 An \textit{$n$-reset limit cycle} $L(v,w,n)$ (or just \textit{$n$-reset cycle}), where $n$ is a positive integer, is a subset of the state space of system \eqref{eq:modelgen}(or \eqref{eq:modelgenbis})--\eqref{eq:reset} such that $L$ is an $n$-reset periodic cycle and a limit set of system \eqref{eq:modelgen}(or \eqref{eq:modelgenbis})--\eqref{eq:reset}. 
\end{definition}
\begin{remark}
Given that the reset is instantaneous, in Definition~\ref{def:nrespercyc} we have used the time interval in which we count the time instances of resets for $t\in (0, \, T]$. This ensures that if the initial $v_0$ is equal to $v_{\mathrm{thr}}$, then we do not count it as reaching the threshold twice. Note that if we chose $t\in [0, \, T)$ then if $v_0  = v_{\mathrm{res}}$ we would miss out one reset, and hence the latter choice of the time interval is not used in Definition~\ref{def:nrespercyc}.     
\end{remark}
\begin{remark}
Since we consider only positive times, an $n-$reset cycle according to Definition~\ref{def:nrescyc} would imply an attractor since only $\omega-$limit sets would be possible in the model system. However, unstable $n$-reset periodic orbits are also possible in system \eqref{eq:modelgen}(or \eqref{eq:modelgenbis})--\eqref{eq:reset}. Thus an $n$-reset cycle in the text will refer to either  a stable or unstable $n$-reset periodic cycle.     
\end{remark}

We first present numerical evidence that system~\eqref{eq:modelgen}--\eqref{eq:reset} can undergo spike-adding transitions within very narrow parameter intervals. Figure~\ref{fig:adding} depicts a transition between $2$- and $3$-reset cycles, obtained by direct simulation and which is occurring under parameter variation in the system, for $F(v,w) = -w$. This is a good incentive to look at this system with analytical tools in order to decipher the mechanisms underpinning such complex dynamics happening within such a narrow parameter band. In particular, we will consider how the asymptotic dynamics changes when we vary parameter $k$. The two simulations shown in Figure~\ref{fig:adding} are performed for the following fixed parameter values: $I=0.1$, $\eps=0.05$, $b=0$, $v_{\mathrm{res}}=0.2$,  $v_{\mathrm{thr}}=1$, and for a minute variation of parameter $k$ giving rise to three different periodic attractors with canard segments. 
Although the difference in the value of $k$ is remarkably small, the difference in the asymptotic dynamics is of order 1: in panel (a) ($k=0.15033$), the dynamics is attracted towards a stable limit cycles with three fast oscillations or \textit{spikes} whereas in panel (b) ($k=0.15034$) it is attracted towards a stable limit cycle with four spikes; finally, the periodic attractor displayed in panel (c) ($k=0.15037$) has two spikes per period. Note that spikes here are discontinuous, even though in the figure, the cycles appear as continuous curves for convenience.

\begin{figure}[!t]
\centering
\includegraphics[width=\textwidth]{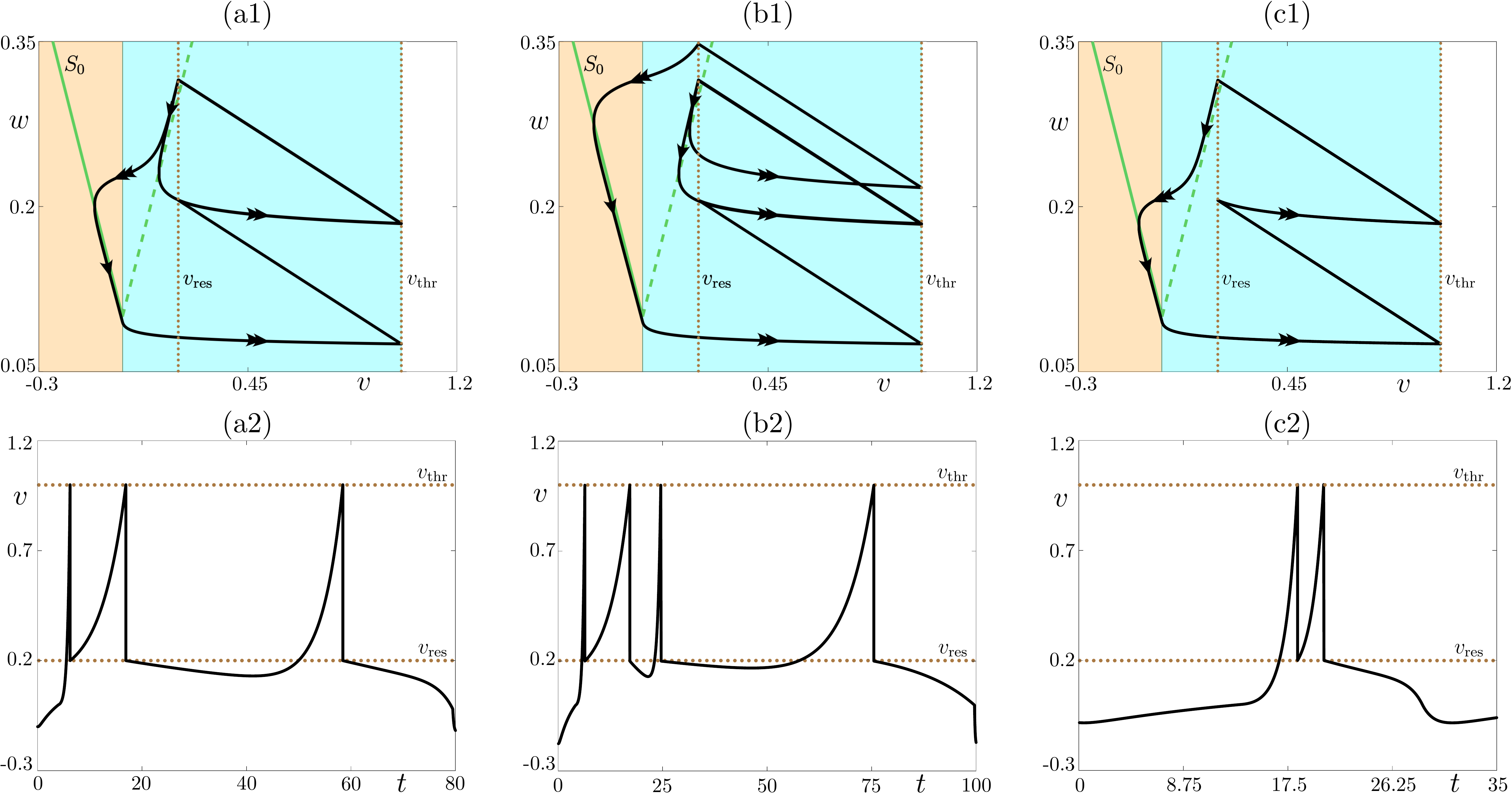}
\caption{Simulation of system \eqref{eq:modelgen}--\eqref{eq:reset} with $F(v,w)=-w$, and for the following fixed parameter values: $I=0.1$, $\eps=0.05$, $b=0$, $v_{\mathrm{res}}=0.2$,  $v_{\mathrm{thr}}=1$. We set $k$ to: (a) $k=0.15033$, which leads to a three-reset cycles; (b) $k=0.15034$, which leads to a four-reset cycle; (c) $k=0.15037$, which leads to a two-reset cycle. Top panels show the solution in the phase plane together with the critical manifold $S_0$, the reset line $\{v=v_{\mathrm{res}}\}$ and the threshold line $\{v=v_{\mathrm{thr}}\}$; bottom panels show the time profile for $v$ over one period, together with reset and threshold lines.}
\label{fig:adding}
\end{figure}
%
%---
\subsection{Analysis through reset-induced canard cycles}
\label{sec:cit}
%---
In the current section, we will explain the spike-adding mechanism presented above through direct simulations. To this aim, we will first prove the existence of periodic cycles with $N$ resets, for any integer $N  \geq 1$. 
\subsubsection{Existence of cycles with resets}
Consider system~\eqref{eq:modelgen}--\eqref{eq:reset} with $b  = 0$, and parameters $v_{\mathrm{res}}$, $I$, $k$ and $v_{\mathrm{thr}}$ positive. We first rewrite the system in a piecewise manner as
\beq
f^+ :\begin{cases} v^\prime = v - w + I,\\
w^\prime = -\eps w,\quad\mbox{for}\quad 0 \leq v < v_{\mathrm{thr}},  \end{cases}
\label{eq:vfp}
\eeq
\beq
f^- :\begin{cases} v^\prime = -v - w + I,\\
w^\prime = -\eps w,\quad\mbox{for}\quad  v < 0,  \end{cases}
\label{eq:vfm}
\eeq
and the reset condition as
\beq
R:\quad (v,\, w) \mapsto (v_{\mathrm{res}},\,w + k)\quad\quad\mbox{for}\quad\quad v = v_{\mathrm{thr}}.
\label{eq:rst}
\eeq
The flow solutions corresponding to $f^\pm$ are given by 
\beq
\phi_+: \begin{cases} v(t) = -I + \frac{\ds w_0}{\ds 1+ \eps }(\exp(-\eps t) - \exp(t)) + \left(v_0 + I \right)\exp(t),\\
w(t) = w_0\exp(-\eps t), 
\end{cases} 
\label{eq:flsp}
\eeq
and
\beq
\phi_-: \begin{cases} v(t) = I + \frac{\ds w_0}{\ds 1 -\eps}(\exp(-t) - \exp(-\eps t)) + (v_0 - I)\exp(-t),\\
w(t) = w_0\exp(-\eps t). 
\end{cases} 
\label{eq:flsm}
\eeq
The attracting and repelling parts of the critical manifold of the system are defined as: 
\beq\label{eq:crM}
\begin{split}
S_0^+ &= \{(v,\, w)\in \mathbb{R}^+\cup\{0\}\times\mathbb{R}\,\,:\,\, w = v + I\}\\
S_0^- &=\{(v,\, w)\in \mathbb{R}^-\times\mathbb{R}\,\, :\,\, w = -v + I\},
\end{split}
\eeq
respectively. Likewise, attracting and repelling slow manifolds of the full system are defined as:
\beq\label{eq:slM}
\begin{split}
S_\eps^+ &= \{(v,\, w)\in \mathbb{R}^+\cup\{0\}\times\mathbb{R} \,\, :\,\, w = (1+\eps)( v + I)\}\\
S_\eps^- &=\{(v,\, w)\in \mathbb{R}^-\times\mathbb{R}\,\, :\,\, w = (\eps - 1)(v - I)\},
\end{split}
\eeq
respectively; see~\cite{prohens2013} for details.
\begin{proposition}
Consider system~\eqref{eq:vfp}-\eqref{eq:rst} with positive parameters $v_{\mathrm{thr}} = {\mathcal{O}}(1)$, $I = \mathcal{O}(v_{\mathrm{res}})$. Then for every $\eps>0$ small enough and $w_0 > (1+\eps)(v_{\mathrm{res}} + I)$, there exists $k > 0$ such that the system possesses a limit cycle characterised by $N\geq 1$ resets and a periodic point $(v_{\mathrm{res}}, w_0)$.  
\end{proposition}
\begin{proof}
Starting from given point $(v_0, w_0) = (v_{\mathrm{res}}, w_0)$ the system evolves to some point $(0, w_I)$ where 
$w_I = w_0\exp(-\eps t_1)$. Note that $v^\prime$ is negative at the initial point and it remains negative following all points along the trajectory until it can become $0$ on $S_0^-$. So, there must exist $t_1 > 0$ when the trajectory reaches $\{v = 0\}$.\newline At $(0,w_I)$, $v^\prime < 0$ as this is not the set $S_0$ and so the system will follow $\phi_-$ until it crosses $S_0^-$ at some point $(v_1,w_1)$, and then the trajectory will start mowing back to the right until some point $(0,w_{II})$ is reached. Note that there must exist some time $t_2 > 0$, which is the flight time from $(0,w_I)$ to $(v_1,w_1)$, since at all points along the trajectory, from $(0,w_I)$ to $(v_1,w_1)$, $v^\prime \leq 0$ (it is $0$ only at $(v_1,w_1)$) and $w^\prime < 0$. Similarly, there must exist some time $t_3 > 0$, which is the flight time from $(v_1,w_1)$ to $(0,w_{II})$ since at all points along the trajectory, from $(v_1,w_1)$ to $(0,w_{II})$, $v^\prime \geq 0$ (it is $0$ only at $(v_1,w_1)$) and $w^\prime < 0$. Furthermore $0 < w_{II} < I$ since $w_{II} = w_0\exp(-\eps (t_1 + t_2 + t_3))>0$ and the trajectory  crossed $S_0^-$ only at $(v_1,w_1)$. \newline
Consider first the case $N = 1$. At $(0,w_{II})$, $v^\prime > 0$ and it remains so along the trajectory until $v_{\mathrm{thr}}$ is reached at $v_{III} = v_{\mathrm{thr}},\, w_{III} = w_0\exp(-\eps(t_1 + t_2 + t_3 + t_4))$. Again, the existence of $t_4 > 0$ is guaranteed since $v^\prime > 0$ along the trajectory. 
 At this point we apply the reset map~\eqref{eq:rst} and obtain the final point $v_{IV} = v_{\mathrm{res}},\, w_{IV} = w_0\exp(-\eps(t_1 + t_2 + t_3 + t_4)) + k$. For a periodic cycle to exist, we require that there exists some value of $k > 0$ which 
satisfies 
$$
w_0 = \frac{k}{1 - \exp(-\eps(t_1 + t_2 + t_3 + t_4))} > (1+\eps)(v_{\mathrm{res}} + I).
$$
Using vector fields $f^\pm$, we may show that $t_1 + t_2 + t_3 + t_4 = \mathcal{O}(1/\eps)$. We may then set $k = (1+\eps)(v_{\mathrm{res}} + I)$, and the inequality above is satisfied. 

Consider now the case $N = 2$. All the steps are the same until the point $w_{IV} \not = w_0$ is reached. In this case, the inequality condition $w_{IV} > (1+\eps)(v_{\mathrm{res}} + I)$ is violated and from $(v_{\mathrm{res}}, w_{IV})$ the trajectory evolves towards threshold $v_{\mathrm{thr}}$. Again, $v^\prime$ is positive along the trajectory segment from $(v_{\mathrm{res}}, w_{IV})$ until the threshold is reached and the existence of $t_5 > 0$, which is the flight time from 
$(v_{\mathrm{res}}, w_{IV})$ to $(v_{\mathrm{thr}}, w_{V})$, is guaranteed with 
$$
w_V = w_0\exp(-\eps(t_1 + t_2 + t_3 + t_4 + t_5)) + k \exp(-\eps(t_5)).  
$$
At $(v_{\mathrm{thr}}, w_{V})$, the reset map is applied, which yields the final point 
$$
v = v_{\mathrm{res}}, \quad w_F = w_0\exp(-\eps(t_1 + t_2 + t_3 + t_4 + t_5)) + k \exp(-\eps(t_5)) + k.
$$
For the $2-$reset periodic cycle to exist, we require
$$
w_0 = \frac{k(1+ \exp(-\eps(t_5)))}{1 - \exp(-\eps(t_1 + t_2 + t_3 + t_4 + t_5))} > (1+\eps)(v_{\mathrm{res}} + I) 
$$
and 
$$
\frac{k(1+ \exp(-\eps(t_5)))\exp(-\eps(t_1 + t_2 + t_3 + t_4))}{1 - \exp(-\eps(t_1 + t_2 + t_3 + t_4 + t_5))} + k < (1+\eps)(v_{\mathrm{res}} + I). 
$$
The second inequality simplifies to
$$
\frac{k(1 + \exp(-\eps(t_1 + t_2 + t_3 + t_4)))}{1 - \exp(-\eps(t_1 + t_2 + t_3 + t_4 + t_5))} < (1+\eps)(v_{\mathrm{res}} + I). 
$$
Without loss of generality, let us assume that point $w_{IV}$ is at least $\mathcal{O}(\eps)$ away from the slow manifold $S_\eps^+$, that is $(1+\eps)(v_{\mathrm{res}} + I) - w_{IV} > \eps$. In this case one can verify that $t_5 = \mathcal{O}(1)$. 

Let $D = 1 - \exp(-\eps(t_1 + t_2 + t_3 + t_4 + t_5))$ and $T = (1+\eps)(v_{\mathrm{res}} + I)$. One may then write the two inequalities as  
$k(1 + \exp(-\eps(t_1 + t_2 + t_3 + t_4))) < TD - \eps D$ and $k(1+ \exp(-\eps(t_5))) > TD$.
Since as before, $t_1 + t_2 + t_3 + t_4 = \mathcal{O}(1/\eps)$ and $t_5 = \mathcal{O}(1)$,
it is clear that $(1 + \exp(-\eps(t_1 + t_2 + t_3 + t_4))) < (1+ \exp(-\eps(t_5)))$, and some $k > 0$ may be chosen ensuring that both inequalities are satisfied. For example, one may choose $k = 3/5TD$. This completes the case for $N = 2$. 

We now turn to the existence of an $N$-cycle, by which is meant a cycle with $N$ resets for $N \geq 3$. Define first the values of variable $w$ immediately after resets, that is when $v$ reaches $v_{\mathrm{thr}}$ and is set to $v = v_{\mathrm{res}}$. Thus, let $w_i$ for $i = 1,\,2,\,\ldots,\,N$ correspond to the values of $w$ at times $t_1+t_2+t_3 +,\,\ldots,\,+t_{i + 3}$, for $i = 1,\,\ldots,\,N$, where $T_i = t_1+t_2+t_3 +,\,\ldots,\,+t_{i + 3}$ refers to the time of $i$th reset. Specifically, 
\begin{equation}
w_i = w(T_i)\,\,\,\mbox{when}\,\,\,\lim_{t\downarrow T_i}v(t) = v_{\mathrm{thr}}.
\label{eq:rstW}
\end{equation} 
We then have:
\begin{equation}\label{eq:w1wN}
\begin{split}
 w_1 & = \mbox{e}^{-\eps(t_1 + t_2 + t_3 + t_4)}w_0 + k,\\
 w_2 & = \mbox{e}^{-\eps(t_1 + t_2 + t_3 + t_4 + t_5)}w_0 + k\mbox{e}^{-\eps t_5} + k,\\
        & \vdots\\
w_N & = \mbox{e}^{-\eps(t_1 + t_2 + t_3 + t_4 + t_5 + \ldots + t_{N + 3})}w_0 + k\mbox{e}^{-\eps (t_5 + t_6 + \ldots + t_{N + 3})} + \ldots\\ 
         & \quad k\mbox{e}^{-\eps (t_6 + t_7 + \ldots + t_{N + 3})} + \ldots + k\mbox{e}^{-\eps t_{N + 3}} + k.
\end{split}
\end{equation}
As in the case for $N = 1$ or $N = 2$, using vector fields $f^\pm$ we can show that $t_1 + t_2 + t_3 + t_4 = \mathcal{O}(1/\eps)$. Without loss of generality, let us suppose that $w_{N-1}$ is at least $\mathcal{O}(\eps)$ away from $S_\eps^+$, that is $(1+\eps)(v_{\mathrm{res}} + I) - w_{N - 1} > \eps$. Using vector field $f^+$, we may then show that 
$t_{i + 3} = \mathcal{O}(1)$ (for $i = 2,\ldots, N$) and if $k > t_{N+3}(v_{\mathrm{res}} + I)\eps + \mathcal{O}(\eps^2)$ we have that:
\beq
t_1 + t_2 + t_3 + t_4 > t_{N+3} > t_{N+2} > t_{N + 1} > \ldots > t_5, 
\eeq
and 
\beq
w_N > w_{N - 1} > \ldots > w_2 > w_1.
\label{eq:inseq}
\eeq
For the existence of an $N$-reset cycle, we require 
\beq
w_N = w_0 >(1+\eps)(v_{\mathrm{res}} + I) > w_{N - 1} > \ldots > w_2 > w_1.
\eeq
It implies that 
\beq\label{eq:exst}
\begin{split}
w_0 & = k\frac{ \mbox{e}^{-\eps (t_5 + t_6 + \ldots + t_{N + 3})} + \mbox{e}^{-\eps (t_6 + t_7 + \ldots + t_{N + 3})} + \ldots + \mbox{e}^{-\eps t_{N + 3}} + 1}{1-\mbox{e}^{-\eps (t_1 + t_2 + \ldots + t_{N + 3})}} >  \ldots\\
        &\quad (1+\eps)(v_{\mathrm{res}} + I) > w_{N-1} > w_{N-2} >\ldots > w_1.
\end{split}
\eeq
Clearly, $k > t_{N+3}(v_{\mathrm{res}} + I)\eps + \mathcal{O}(\eps^2)$ may be chosen so that condition~\eqref{eq:inseq}
holds and by appropriately adjusting the threshold parameters $(1+\eps)(v_{\mathrm{res}} + I)$ one may ensure that 
condition~\eqref{eq:exst} holds as well. This completes the proof. 
\end{proof}
We should note here that the number of resets is a function of system parameters, and the structure unraveled by Proposition 1 guarantees that at least one periodic cycle with some finite number of resets may be found. However, it does not touch upon whether or not there is an upper bound on the number of resets. 
\subsubsection{Existence of canard cycles with resets}
\begin{proposition}
Consider system~\eqref{eq:vfp}-\eqref{eq:rst} with positive parameters $v_{\mathrm{thr}} = {\mathcal{O}}(1) \gg v_{\mathrm{res}}$, $I = {\mathcal{O}}(v_{\mathrm{res}})$. Then for every $\eps>0$ small enough and $w_0 = (1+\eps)(v_{\mathrm{res}} + I)$, there exists $k > 0$ such that the system possesses an $N$-reset periodic cycle and a periodic point $(v_{\mathrm{res}}, w_0)$ with the following set of inequalities~\eqref{eq:exst}: 
\beq
w_0 =  (1+\eps)(v_{\mathrm{res}} + I) = w_{N} > w_{N-1} > w_{N-2} >\ldots > w_1,
\label{eq:exstCN}
\eeq
where $w_i$, $1\leq i\leq N-1$, are the values of the $w$-component of the system immediately after the $i$th reset as
 defined by equation~\eqref{eq:rstW}. Such a periodic cycle is a canard cycle. 
\end{proposition} 
\begin{proof}
Directly from Proposition 1, it follows that there exists an $N$-reset periodic cycle with periodic point
$(v_0,w_0) = (v_{\mathrm{res}}, w_0)$, where  $w_0 > (1+\eps)(v_{\mathrm{res}} + I)$. Without loss of generality, let us assume that we chose the parameters such that $w_0 - (1+\eps)(v_{\mathrm{res}} + I) < \delta$, with $ 0 < \delta \ll 1$. Then under small decrease of $k$ all the conditions for the existence of $N$-reset cycle are satisfied, and by the continuity we will have that  $w_0 = (1+\eps)(v_{\mathrm{res}} + I)$ for $k < k^*$. 
Since $w_0 = (1+\eps)(v_{\mathrm{res}} + I)\in S_\varepsilon^+$ then a segment of the cycle lies on the 
repelling part of the slow manifold, and hence the $N$-reset cycle is an $N$-reset canard cycle.   
\end{proof}
\begin{proposition}
Consider system~\eqref{eq:vfp}-\eqref{eq:rst} with parameters  $v_{\mathrm{thr}} = {\mathcal{O}}(1) \gg v_{\mathrm{res}}$, $I = {\mathcal{O}}(v_{\mathrm{res}})$, and some $k^*$, all positive. Let us further assume that a canard cycle with $N = 2$ number of resets as described by Proposition 2 exists for a given set of parameter values. Then for some $k$, sufficiently close to but different from $k^*$, there exists a periodic cycle in the system characterised by $N + 1$ of resets, and a periodic point $(v_{\mathrm{res}}, w_0)$ such that the inequality 
\beq
w_0 =  (1+\eps)(v_{\mathrm{res}} + I) = w_{N + 1} > w_{N} > w_{N-1} > w_{N-2} >\ldots > w_1
\label{eq:exstCNp1}
\eeq
holds. Moreover, we have that $|w_{N+1}-w_N| \ll \eps$. Such a periodic cycle is a canard cycle.
\end{proposition} 

\begin{proof}
If there exists a canard cycle with two resets as described in Proposition 2, then we have 
\begin{equation}
w_0^* = (1 + \eps)(v_{\mathrm{res}} + I) = k^*\frac{(1 + \mbox{e}^{-\eps t_5})}{1 - \mbox{e}^{-\eps(t_1 + \ldots + t_5)}} > 
k^*\frac{(1 + \mbox{e}^{-\eps t_5})\mbox{e}^{-\eps(t_1+\ldots+t_4)}}{1 - \mbox{e}^{-\eps(t_1 + \ldots + t_5)}}.
\label{eq:Ccc}
\end{equation} 
Moreover, 
$$
w_{II} = w_0^*\mbox{e}^{-\eps(t_1 + t_2 + t_3)},
$$
where point $w_{II}$ is the point of crossing $\{v = 0\}$ by flow $\phi_-$. Times $t_1$, $t_2$ and $t_3$ are the  flow times from initial point $(v_{\mathrm{res}},\,(1+\eps)(v_{\mathrm{res}} + I))$ until the point of crossing $\{v = 0\}$ by flow $\phi_-$ (the times are defined in the proof of Proposition 1). We know that $t_1 + t_2 + t_3 = \mathcal{O}(1/\eps)$. Note that $k^* = \mathcal{O}(v_{\mathrm{res}})$ for equation~\eqref{eq:Ccc} to hold. 
From~\eqref{eq:Ccc}, we further have that 
\beq
k^* = \frac{(1 + \eps)(v_{\mathrm{res}} + I)(1 - \mbox{e}^{-\eps(t_1 + \ldots + t_5)})}{(1 + \mbox{e}^{-\eps t_5})}.
\eeq
Consider now some some $k \not = k^*$ sufficiently close to $k^*$, and such that we have 
$$
w_3 = w_0^* = (1+ \eps)(v_{\mathrm{res}} + I) > w_2 = (1+ \eps)(v_{\mathrm{res}} + I) - \delta(k) = 
w_{II}\mbox{e}^{-\eps (t_4^\prime + t_5^\prime)} + k\mbox{e}^{-\eps t_5^\prime} + k,
$$ 
where $w_1 = w_{II}\mbox{e}^{-\eps (t_4^\prime)} + k$ and $w_2 = w_{II}\mbox{e}^{-\eps (t_4^\prime + t_5^\prime)} + k\mbox{e}^{-\eps t_5^\prime} + k$. 

Since $k$ is sufficiently close to $k^*$, $t_4^\prime = \mathcal{O}(t_4)$ and $t_5^\prime = \mathcal{O}(t_5)$, and then $w_2 > w_1$. We now seek to find some time $t_6(k) > 0$ after which point $(v_{\mathrm{res}},w_2)$ under flow $\phi_+$, and a reset, is mapped onto $w_0 = (1+ \eps)(v_{\mathrm{res}} + I)$ for a suitable $k \not = k^*$ and sufficiently close to $k^*$. Let us suppose that such a $k$ exists. Then, it must be equal to 
\beq
k = \frac{(1 + \eps)(v_{\mathrm{res}} + I)(1 - \mbox{e}^{-\eps(t_1 + \ldots + t_4^\prime + t_5^\prime + t_6)})}{(1 + \mbox{e}^{-\eps t_5^\prime} +  \mbox{e}^{-\eps (t_5^\prime + t_6)})}
\label{eq:valofk}
\eeq
and $t_6(k)$ is a solution of 
$$
1 = -I + (v_{\mathrm{res}}+I)\mbox{e}^{-\eps t_6} - \frac{\delta(k)}{1+\eps}\mbox{e}^{-\eps t_6} + 
 \frac{\delta(k)\mbox{e}^{t_6}}{1+\eps}.
$$
For $\delta(k)$ sufficiently small, that is for $\delta(k)  = |w_3-w_2| < \eps(v_{\mathrm{res}}+I) \ll \eps$ there exists the solution $t_6(k)$ such that 
$$
\mathcal{O}(\delta(k)\mbox{e}^{t_6}) = \mathcal{O}(1).
$$ 
On the time interval $[0,t_6]$, the velocity of the $v$-component is equal to $\dot v = -\eps(v + I) + \delta(k)(1+\eps)$, thus for $\delta(k)$ sufficiently small and hence for $k$ appropriately chosen, the average value of the velocity may be such $t_6 = \mathcal{O}(v_{\mathrm{res}}/\eps)$ and $\mathcal{O}(\delta(k)\mbox{e}^{t_6}) = \mathcal{O}(1)$. This ensures the existence of a canard cycle with $3$-resets.  
\end{proof}
Note that we are not able to determine the sign of $k^*-k$, that is, whether it is by increasing or decreasing the bifurcation parameter that we will find the $3$-reset canard cycle. Moreover, we conjecture that if there exists a canard cycle with $N > 2$ number or resets, as described in Proposition 2, then in some sufficiently small neighbourhood of the bifurcation parameter, there exist canard cycles with $N+1, N+2,\cdots, 2N$ resets in system~\eqref{eq:vfp}-\eqref{eq:rst}. 
This will be further explained in the following section.     
\subsubsection{Transition from 2-reset to 3-reset cycles}
We now explain the 2- to 3-reset transition of the periodic cycle observed in our model example. 
We shall start by considering a 2-reset periodic cycle existing for some $k^*$, all other parameters being fixed. Thus, we first assume that Proposition 1 holds for $N = 2$.

We will now construct a return map from a segment on $\{v = v_{\mathrm{res}}\}$ back to itself under the action of flows $\phi_\pm$ and the reset map. We have a 2-reset periodic cycle with periodic point, say $w_0^*$ on $\{v = v_{\mathrm{res}}\}$. We choose $k^*$ and all other system parameters such that we have periodic point $w_0^*$ some distance $\zeta = \zeta(k)$ away from the point of intersection of the slow manifold $S_\eps^+$ with $\{v = v_{\mathrm{res}}\}$. That is $w_0^* - \zeta = (1+\eps)(v_{\mathrm{res}} + I)$. Moreover, we choose $k^*$ such that the 2-reset limit cycle is stable. By considering explicit flow solutions~\eqref{eq:flsp} and~\eqref{eq:flsm}, which allow us to determine the variation matrices for the flows, and discontinuity maps, one can show that the nontrivial Floquet multiplier of the 2-reset cycle is mainly determined by the exponential terms multiplying the initial point $w_0^*$, that is the expansion is dominated by the term $\exp(t_e)$ (where $t_e$ is the total time of evolution following flow $\phi_+$) and contraction by the term $\exp(-t_c)$ (where $t_c$ is the time of evolution following flow $\phi_-$). Hence, for $k$ appropriately chosen, we will have  $t_c > t_e$, from which we can ensure local contraction and so the 2-reset cycle is an attractor. %
The effect of the discontinuity matrices on the stability may be considered negligible.
   
Consider now a local return map, which maps an interval containing $w_0^*$, say $I$, back to itself, with $I$ corresponding to a segment of Poincar\'e section
$\Pi = \{ v = v_{\mathrm{res}}, \, w = (w_0^*-\xi, w_0^*+\xi) \}$, where $\zeta > \xi > 0$ and $\xi$ is  sufficiently small.  
The return map, say $P_c$, such that $P_c:\Pi \mapsto \Pi$ is defined as 
\begin{equation}\label{eq:RTMap}
\begin{aligned}
P_c &:\; (v_{\mathrm{res}},w_n)\mapsto (v_{\mathrm{res}},w_{n+1}),\;\text{with}\\
w_{n+1} &= R(P(\phi_+(R(P(\phi_+(P(\phi_-(P(\phi_+((v_{\mathrm{res}},w_n),\cdot)),\cdot)),\cdot))),\cdot))),
\end{aligned}
\end{equation}
where $n\in\mathbb{N}$ and the dot symbol corresponds to an appropriate time of evolution which brings the flow solution of the 2-reset cycle to either $\{v = 0\}$ or $\{v = v_{\mathrm{res}}\}$. The projection mapping $P$ is a standard projection (along the flow lines generated by vector fields $f_{\pm}$) onto switching set $\{v = 0\}$, or threshold set $\{v = v_{\mathrm{res}}\}$, which is non-identity for points in $I$ different from $w_0^*$. Finally, $R$ denotes the reset map~\eqref{eq:reset}.

Since we choose $k^*$ such that the 2-reset limit cycle is stable, it then follows that for $\xi$ sufficiently small $P_c$ is a contraction for all points in $I$. Moreover, it may be further verified that $P_c$ is an injection with reversed orientation. Decreasing $k$ implies that $t_c$ becomes shorter and $t_e$ becomes longer since, by continuity, the periodic point (on the Poincar\'e section $\Pi$) moves towards $(1+\eps)(v_{\mathrm{res}}+I)$ as a function of $k$. This implies that the 2-reset cycle has to undergo a period-doubling bifurcation for some $k<k^*$, and the 2-reset cycle becomes unstable by further decrease in $k$. Further decrease in $k$ implies that the 4-reset cycle born in the period-doubling bifurcation will become a 4-reset maximal canard cycle, and it is either this periodic cycle, or a chaotic attractor which may be born in a period-doubling cascade (on an extremely small parameter variation), which will collide with the 3-reset maximal canard cycle that exits in some neighbourhood of the 2-reset maximal canard cycle as described in Proposition 3. 
By continuity (see Proposition 2), the 2-reset canard cycle becomes a 2-reset maximal canard cycle under variation of $k$, which is extremely unstable due to its segment along the repelling slow manifold $S_\eps^+$.

Starting from the 3-reset maximal canard cycle, again by continuity with respect to $k$, one may expect to see a branch of 3-reset cycles within an open set of parameter values. Symbolically, this scenario could be denoted as $2\leftrightarrow(4)(\cdot)\leftrightarrow 3$, which means that the continuous transition from the 2-reset to the 3-reset cycle has to occur via a 4-reset cycle possibly followed by a chaotic attractor that we represent by $(\cdot)$. The exact scenario depends on the dynamics of a local map around a periodic point of the 3-reset maximal canard cycle.

Let $w^*_{3r}$ denote the periodic point of the 3-reset maximal canard cycle. It is the periodic point on the maximal 3-reset canard cycle on $\{v = v_{\mathrm{res}}\}$, and it is given by
\beq
w^*_{3r} = \phi_+(R(\phi_+(R(\phi_+(P(\phi_-((v_{\mathrm{res}},(1+\eps)I),t_2)),t_3)),t_4)),t_5),
\label{eq:3rc}
\eeq  
where $t_2$ is the flow time back to the set $\{v = 0 \}$, and the subsequent times $t_3$, $t_4$ and $t_5$, are flow times to the set $\{v = v_{\mathrm{thr}}\}$ along the 3-reset maximal canard cycle. Finally, $R$ denotes the reset map~\eqref{eq:reset}.

We now define a segment $J$, which will serve us as a local Poincar\'e section such that $J = (w^*_{3r}-\eta,\,w^*_{3r}+\eta)$, with $\eta$ sufficiently small. A local return map $P_{3RC}$, where $P_{3RC}:J\to J$, $w_n\mapsto w_{n+1}$, is then given by
\begin{equation}\label{eq:3rcmap}
\begin{aligned}
w_{n+1} &= R(P(\phi_+(P(\phi_-(P(\phi_+(R(P(\phi_+((v_{\mathrm{res}},w_n),t_0))),t_1)),t_2)),t_3)))\\
 & ...\;\text{for}\;w_n \geq w^*_{3r},\\
w_{n+1} &= R(P(\phi_+(R(P(\phi_+((v_{\mathrm{res}},w_n),t_4))),t_5)))\\
 & ...\;\text{for}\;w_n < w^*_{3r}.
\end{aligned}
\end{equation}

The map $P_{3RC}$ is a continuous, piecewise-linear injection with reversed orientation to leading order in $w_n$. The dynamics of piecewise-linear continuous maps has been extensively studied, among others, in e.g.~\cite{avrutin2011,avrutin2006,avrutin2013,banerjee2000,schenke2011}.  
Note that additional flow time flows $t_0$ and $t_1$ are solution times along the 3-reset maximal canard cycle from point $(v_{\mathrm{res}},\,w^*_{3r})$ to $v= v_{\mathrm{thr}}$ and from $(v_{\mathrm{res}},(1+\eps)(v_{\mathrm{res}}+I))$ to $(0,\,(1+\eps)I)$, respectively. The other times are the same as in~\eqref{eq:3rc}. Finally, $R$ is, as before, the reset map~\eqref{eq:reset}, and $P$ is a standard projection map along the flow lines. 
 
The above scenario will also occur for reset cycles with a higher number of resets. Note, though, that to prove such cases an additional care must be taken, when considering the continuity argument for the existence of periodic points with respect to parameter variation, and the existence of local Poincar\'e maps for open intervals. Moreover, a transition, say from a $3$- to $4$-reset cycles will occur via a continuous transition through a 6- and possibly a 5-reset cycle (and maybe a chaotic attractor), or a 6-reset cycle (and then a chaotic attractor). In the notations previously introduced, this would be described as $3\leftrightarrow(6)(\cdot) \leftrightarrow 4$. 

%----------------------------------------------------
\section{Numerical bifurcation approach}
\label{sec:nbif}
%----------------------------------------------------
Canard solutions will occur in model system (\ref{eq:modelgen})-(\ref{eq:reset}) with $F(v,w) = -w$ within certain intervals of paremeter values. As we explained in Sec. \ref{sec:cit}, transitons between these solutions organise the spike-adding scenario, and these transitions are accompanied by exponentially-small phenomena that are difficult to capture numerically. We therefore use a boundary-value approach in combination with pseudo-arclength continuation, in order to compute branches of limit cycles and/or parametrised families of an orbit segments.

To this aim, we use the software package \textsc{auto}, which \textit{a priori} cannot deal with nonsmooth systems. However, one can define a higher-dimensional problem with carefully-chosen boundary conditions that effectively will allow to compute families of limit cycles of the system of interest. Our strategy to achieve such computations can be described as follows. First,  we enlarge the problem so that every segment in a linearity zone, and every segment in between resets, becomes solution to one copy of the original system. Second, we define boundary conditions to ensure the match-up between all such orbit segments by accounting for crossings of the switching manifold $\{v=0\}$ as well as for threshold crossings and associated resets. Finally,  we verify the well-posedness of the numerical problem and perform arclength continuation. 

We will now present the numerical results obtained via the above procedure. Namely, we could compute explosive branches of limit cycles of the original problem along spike-adding transitions. There are of course a number of limitations to this approach and we wish to highlight two of them.\newline The first limitation is that it is hard to detect stability and bifurcations without having to implement oneself additional test-functions, for which one would have to substantially complicate the numerical setup, and possibly even edit the \textsc{auto} code. This defeats the purpose of the numerical procedure we introduce here, which is to efficiently compute, within an existing software, branches of limit cycles in a hybrid system like~\eqref{eq:modelgen}-\eqref{eq:reset} undergoing canard-mediated spike-adding transitions, which are quite challenging to find by direct simulations. In order to perform our computations, we need to preprocess the initial periodic solution obtained by direct simulation and write down a well-posed two-point boundary-value problem outlined below. This simplicity is a great advantage of our approach.\newline The second limitation and drawback of our numerical approach with \textsc{auto} is that the two-point boundary-value problem (BVP) that we setup can give spurious solutions. However, we can avoid these by stopping the continuation as soon as the integration time of one segment, forming the overall cycle, becomes negative. Future work may include implementations of a numerical method that would avoid these spurious solutions as well as detect bifurcations and transitions discussed in current work.

Our strategy to compute branches of $N$-reset limit cycles in system~\eqref{eq:modelgen}-\eqref{eq:reset} is therefore to use a boundary-value problem approach in tandem with numerical continuation. Of course, the system is piecewise linear and so one can obtain a lot of information analytically about the families of periodic solutions, including the existence of cycles containing reset-induced canard segments. However, the associated branches of solutions are ``explosive'' in parameter space, and so it may be tedious if not outright impossible to compute such families using direct simulation. In short, a boundary-value approach discussed here is better suited to such a task. 

Given that standard continuation packages like \textsc{auto} do not a priori handle nonsmooth dynamical systems, even more so systems with resets as we mentioned earlier, we need to decompose a typical bursting cycle with $N$ resets into $N+2$ segments along which the system is linear, together with appropriate boundary conditions. There are $N-1$ segments which correspond to parts of the trajectory where $v$ lies between $v_{\mathrm{res}}$ and $v_{\mathrm{thr}}$, and one segment where $v$ lies between $0$ and $v_{\mathrm{thr}}$. There are two additional segments which correspond to what will always be the initial part of the trajectory in our computations, namely a segment from the reset line $\{v=v_{\mathrm{res}}\}$ until the switching line $\{v = 0\}$ where the trajectory is repelled away from the right branch of the critical manifold, and a second segment from the switching line back to itself, where the trajectory approaches the attracting branch of the critical manifold, follows it until the corner where it switched back to the right linearity zone, and there is enters the next flow segment, see Figure~\ref{fig:burst} for an illustration on a 4-reset cycle. There are two types of boundary conditions here:
\begin{figure}[!t]
\centering
\includegraphics[width=0.5\textwidth]{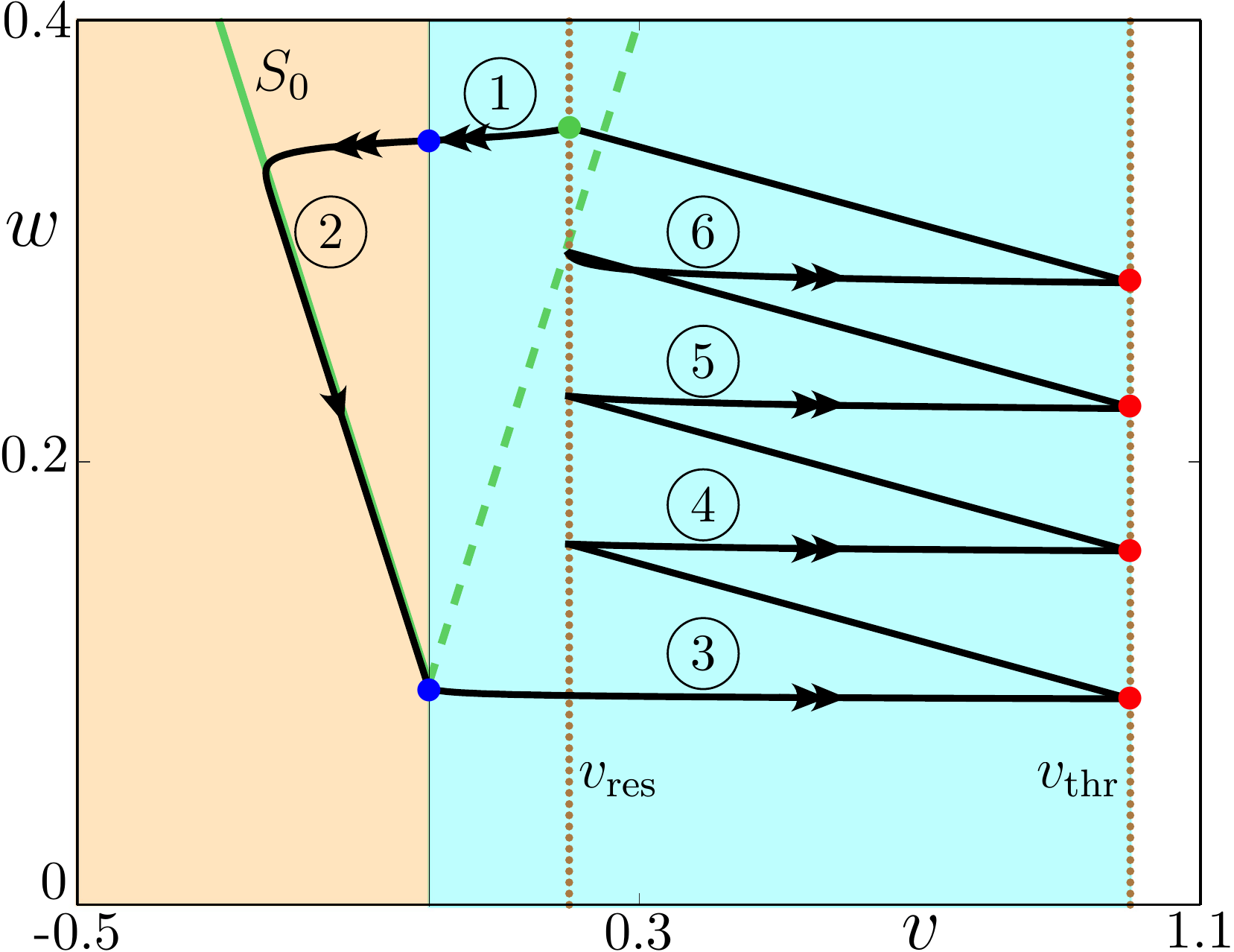}
\caption{Typical bursting cycles of system~(\ref{eq:modelgen})-(\ref{eq:reset}) where $F(v,w) = -w$ and with $N=4$ resets. The cycle can be decomposed into $N+2=6$ segments which are connected to each other via either a boundary or reset conditions. Blue dots correspond to boundary points where continuity conditions must be applied and red dots correspond to boundary points where reset conditions must be applied. Finally, the green dot corresponds to the initial condition on $\{v=v_{\mathrm{res}}\}$.}
\label{fig:burst}
\end{figure}
\begin{itemize}
\item[a.] continuity conditions across the switching line $\{v=0\}$,
\item[b.] reset rules.
\end{itemize}
  A given $N$-reset bursting cycle corresponds to a multi-point boundary-value problem, which is not directly implementable in \textsc{auto}. For this reason, we  transform the original planar system into a $2(N+2)$-dimensional system which corresponds to $N+3$ identical copies of the original system~\eqref{eq:modelgen}-\eqref{eq:reset}. The nonsmoothness across the switching line and the reset rules are implemented as boundary conditions. Therefore, for the numerical problem associated with branches of limit cycles of system~\eqref{eq:modelgen}--\eqref{eq:reset}, we consider the following transformed system:
\begin{equation}\label{eq:nummodel}
\begin{split}
v_i' &= |v_i|-w_i+I\\
w_i' &=\eps(b-w_i),
\end{split}
\end{equation}  
for $i=1\ldots N+3$. Every orbit segment of the original problem along a cycle like the 4-reset bursting cycle shown in Figure~\ref{fig:burst} corresponds to one segment solution of one planar component of the extended system~\eqref{eq:nummodel}. The $N+2$ copies of the original system forming the extended one can be put in any given order, but, for simplicity, we order them according to the order in which each corresponding segment is encountered while moving with the flow (here, counter-clockwise) along the bursting cycle. Consequently, every such segment \circled{i} is a solution to subsystem $(v_i,w_i)$ of the extended system~\eqref{eq:nummodel}. It has an integration time $T_i$ associated with it, and such that: $\sum_i T_i = T$ is the period of the bursting cycle under consideration assuming that the resets are instantaneous. Note that, as is customary in numerical continuation, one rescales time by the integration time of the initial guess of the two-point boundary problem (periodic of not) which allows one to make the integration time an explicit parameter that can be solved for at every computational step. It then follows that, in the rescaled problem, the time vector runs from $0$ to $1$.

 We may now perform a continuation of a two-point boundary value problem of the $2(N+2)$-dimensional system~\eqref{eq:nummodel}, with a number of boundary conditions and a number of free parameters which must be chosen so as to satisfy the well-posedness equation: 
\begin{equation}\label{eq:wellposed}
\mathrm{NBC}-\mathrm{NDIM}+1=\mathrm{NCIP},
\end{equation}
where $\mathrm{NBC}$ corresponds to the total number of boundary conditions, $\mathrm{NDIM}$ is the dimension of the problem and $\mathrm{NICP}$ is the number of free (continuation) parameters that we can vary. It turns out that one can make this numerical problem well-posed and, hence, solvable by \textsc{auto}, by simply freeing all integration times $T_i$, $i=1\ldots N+2$ and one relevant systems parameter, which in our computations will be $k$. That is, we will free $N+3$ parameters in order to compute a one-parameter family of cycles of the original system. This is consistent with the idea that we wish to compute a branch of limit cycles of the original system~\eqref{eq:modelgen}--\eqref{eq:reset} which, in the smooth context, would require to free one system parameter as well as the period of oscillations $T$. In the present case, as $T$ is the sum of all $T_i$, it is natural to free all of them.

\begin{figure}[!t]
\centering
\includegraphics[width=\textwidth]{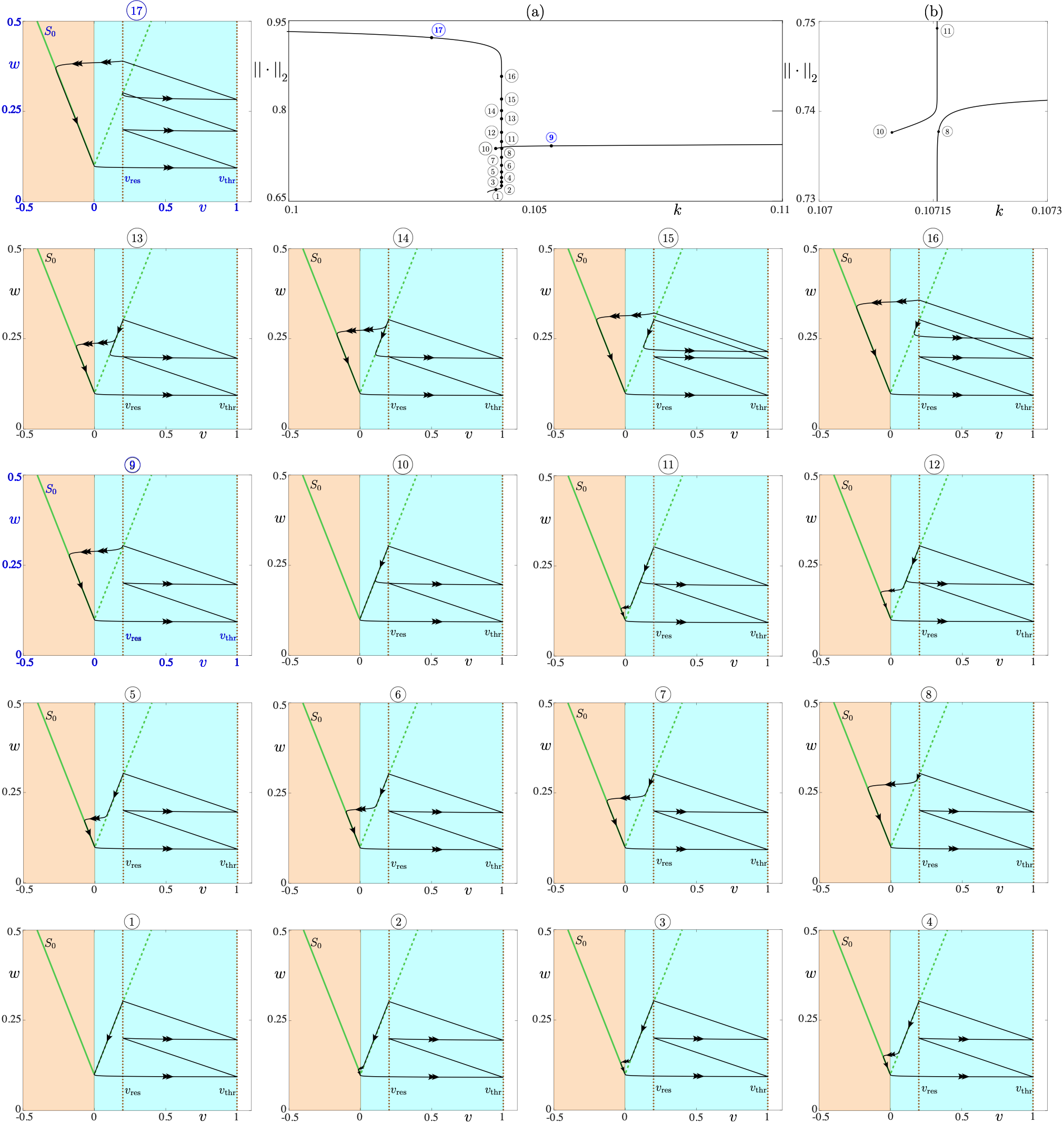}
\caption{Transition between 2-reset cycles and 3-reset cycles via canard solutions. Multiple branches are shown, they have been computed by solving parametrised families of the two-point boundary-value problem~\eqref{eq:bcnd}-\eqref{eq:parbcnd}. Panels (9) and (17) have been highlighted as they correspond to stable 3-reset and 2-reset cycles, respectively. Finally, we remark that the transition between families of 2-reset and 3-reset cycles is discontinuous, as is illustrated in panels (a) and (b).}
\label{fig:trans23}
\end{figure}
Now, the well-posedness equation~\eqref{eq:wellposed} implies that the total number of boundary conditions must be equal to $3(N+2)$. This number is consistent with the fact that the extended system is nothing else than $N+2$ identical copies of the original planar system. To continue an orbit segment, periodic or not, of a planar system in one system parameter with free integration time and one needs three boundary conditions. Hence the number $3(N+2)$ conditions for $N+2$ independent planar systems in which one frees the integration time for each planar independent component and one system parameter shared by all components, e.g. $k$. In the case of a periodic continuation problem in a planar system as ours, to compute a branch of limit cycles one indeed needs to free one system parameter as well as the period, as we mentioned. Then  one needs three conditions: two periodic boundary conditions and one phase condition. Therefore the overall number of conditions that we obtain is consistent. 
These then correspond to three conditions per planar component $(v_i,w_i)$ of the extended system, which we can formulate as follows:

\begin{equation}\label{eq:bcnd}
\begin{split}
v_i(0) &= \alpha_i,\\
w_i(0) &= w_{i-1}(1)+k,\\
v_i(1) &= \beta_i,
\end{split}
\end{equation}
for $i=1\ldots N+2$ with: 
\begin{equation}\label{eq:parbcnd}
\begin{split}
\alpha_2 &=\alpha_3=0,\\
\alpha_i  &=v_{\mathrm{res}},\;i\notin\{2,3\},\\
\beta_1  &=\beta_2=0,\\
\beta_i  &=v_{\mathrm{thr}},\;i\notin\{1,2\},
\end{split}
\end{equation}
and with the convention that $w_0=w_{N+2}$.

This formulation allows us to compute families of solution of the initial system~\eqref{eq:modelgen}-\eqref{eq:reset} with a fixed number of resets through canard transitions. In Figure~\ref{fig:trans23} we show the results of such computations for branches of 2-reset cycles and 3-reset cycles, respectively. The branches do not connect. However they come exponentially close to each other as part of their respective canard regimes. Note that one needs stopping conditions when solving the numerical problem~\eqref{eq:bcnd}-\eqref{eq:parbcnd}, which otherwise can admit spurious solutions. For system~\eqref{eq:modelgen}-\eqref{eq:reset}, the only conditions that need to be implemented which ensure this are: $T_i>0$ at all steps and for $i=1\ldots N+2$ for branches of $N$-reset cycles.

%-------------------------------------------------------------------
\section{AIF model with different adaptation dynamics} 
\label{sec:cplm}
%-------------------------------------------------------------------
We now consider a modified system where the slow equation depends on the fast variable $v$, and so the slow and fast varibales are no longer decoupled. The fast-time formulation is given by:
\begin{equation}\label{eq:model2}
\begin{split}
v' &= |v|-w+I\\
w' &=\eps(v - b),
\end{split}
\end{equation}  
and the slow-time formulation reads
\begin{equation}\label{eq:modelbis2}
\begin{split}
\eps \dot v &= |v|-w+I\\
\dot w &= v - b,
\end{split}
\end{equation}
\begin{figure}[!t]
\centering
\includegraphics[width=0.8\textwidth]{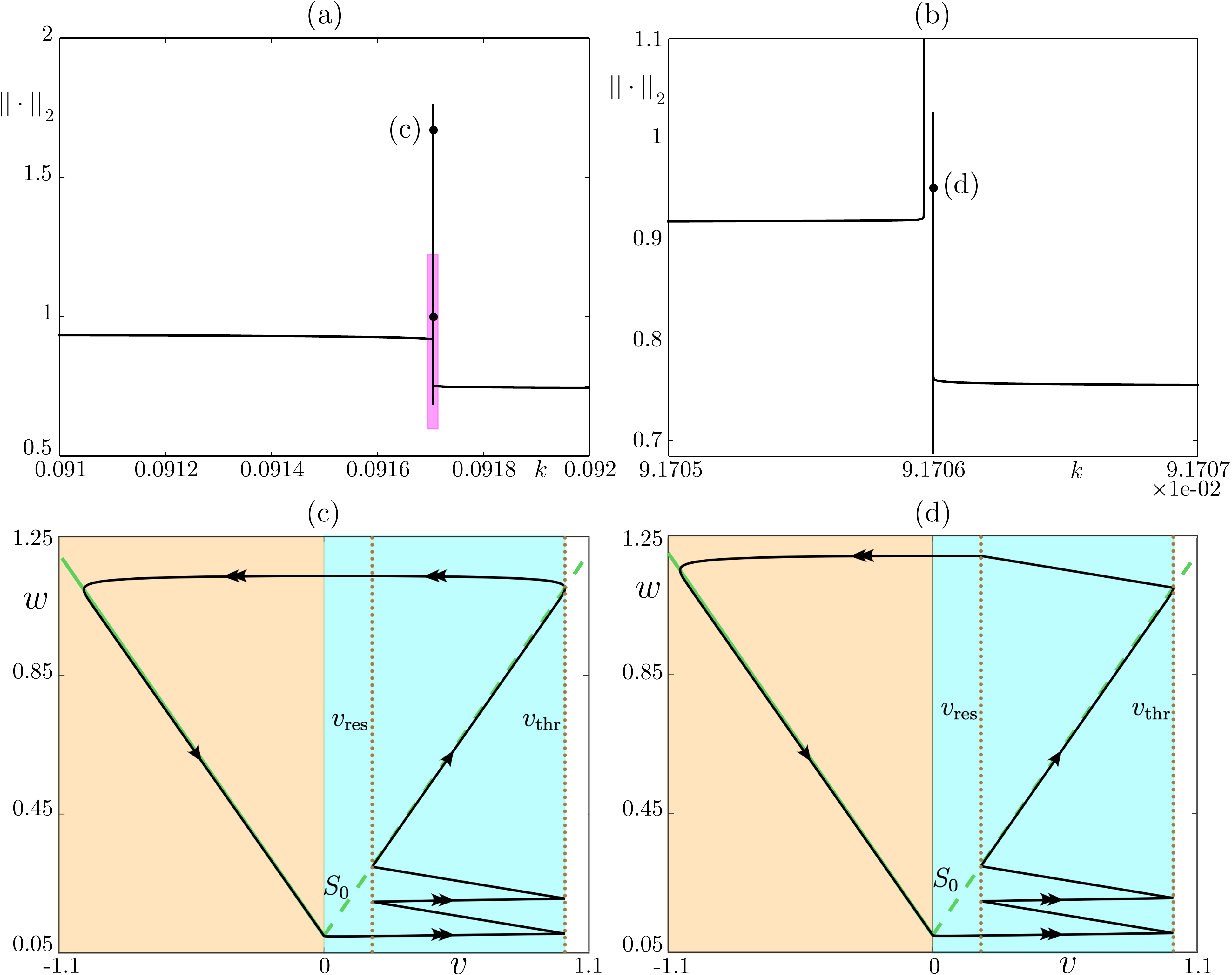}
\caption{Two types of maximal canards in system~\eqref{eq:model2} plotted in the phase plane (lower panels), together with the associated solution branches (upper panels), at the transition between families of 2- and 3-reset cycles. Panels (a) and (c): the maximal canard is the upper limit (in terms of the L$_2$-norm) of the family of headless canard cycles with 2 resets. Panels (b) and (d): the maximal canard is the upper limit of the family of canard cycles with head and with 3 resets (the head corresponding to the third reset).}
\label{fig:maxcan}
\end{figure}
to which we have the same reset conditions as before, namely~\eqref{eq:reset}.
System~\eqref{eq:model2}--\eqref{eq:modelbis2} has the same critical manifold as system~\eqref{eq:modelgen} with $F(v,w) = -w$, the slow nullcline being perpendicular to fast fibers.  

System~\eqref{eq:modelbis2} has a geometry more akin to a van der Pol or a FitzHugh-Nagumo model in the smooth slow-fast context. Together with the reset condition~\eqref{eq:reset}, the dynamics of system~\eqref{eq:modelbis2} is comparable to that of a square-wave burster like the Hindmarsh-Rose (HR) model in the canard-mediated spike-adding regime studied in, e.g.,~\cite{desroches2013}. In particular, the slow flow points upwards along the critical manifold, and this is why the spike-adding process will be more comparable to what was described in~\cite{desroches2013}. Amongst other features, within a one-parameter transition, varying e.g. parameter $k$ as before, one will find family of reset-induced cycles with canard segments. However, due to the particular setup of such systems with reset, there are similarities and differences with the smooth case, as illustrated in Figure~\ref{fig:maxcan}. Panel (a) shows what looks like a continuous branch of limit cycles undergoing an explosive transition. However, when zooming in (panel (b)), one realises that these are in fact two branches with explosive segments that are exponentially close to each other. On the left, a branch of two-reset cycles undergoing a canard explosion whereby, for specific $k$-values within an exponentially-small interval, the dynamics after the second reset follows upwards the slow manifold. Such cycles have an increasingly longer canard segment, up to the maximal canard, which is well-defined as a grazing bifurcation point in this nonsmooth context. On the right, a family of cycles with three-reset each. 

Along the explosive part of the branch, the cycles follow the slow manifold after the second reset, but ``from the other side'' compared to the branch corresponding to the maximal canard cycle in panel (c). In particular, this family of reset-induced canard cycles follow the slow manifold and then escape its vicinity towards the right, hence hitting the threshold line again and having a third reset. Hence we find another maximal canard cycle (panel (d)) which grazes with the threshold line from the other side of the slow manifold than the one shown in panel (c). We conjecture that these different maximal canard cycles, exponentially close in parameter space but well separated in the L$_2$-norm projection, will converge towards each other as $k$ tends to $0$. This is a difference with the smooth case, due to the particular setup, nevertheless the spike-adding process is akin to what is observed in smooth slow-fast systems.

%------------------------------------------------------------------------
\section{Comparison with smooth square-wave bursters}
\label{sec:snsmc}
%------------------------------------------------------------------------
%
\begin{figure}[!t]
\centering
\includegraphics[width=0.8\textwidth]{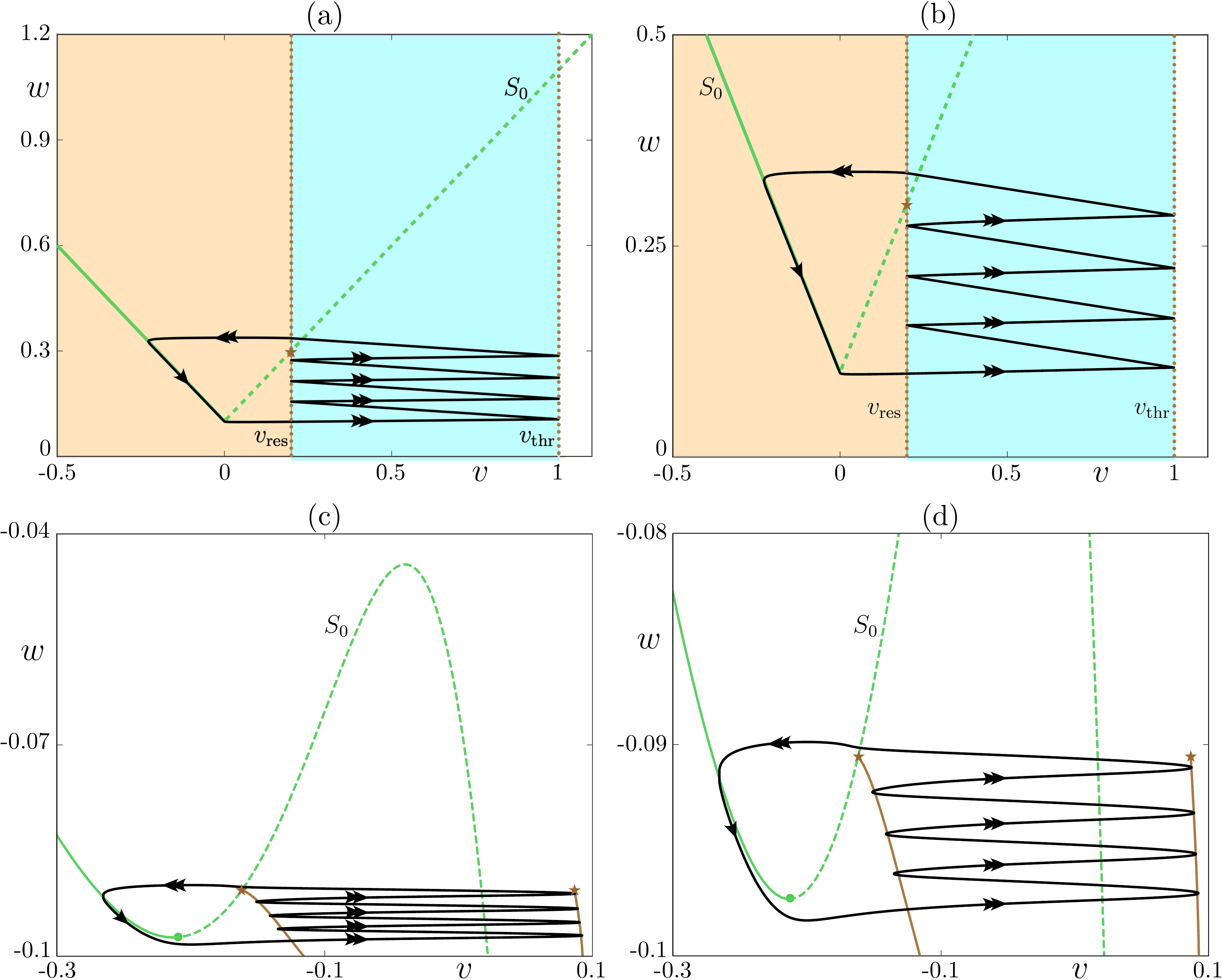}
\caption{Comparison between the nonsmooth systems with resets that we study (panels (a)-(b)), and smooth square-wave bursters like the extended Morris-Lecar model~\cite{terman91} (panels (c)-(d)); right panels are zoomed views of left ones.}
\label{fig:comp}
\end{figure}
The dynamics of system~\eqref{eq:modelgen}--\eqref{eq:reset} with $F(v,w) = -w$ is akin to square-wave bursting, a phenomenon which can be captured within the framework of smooth slow-fast ODEs, with two fast and one slow variables~\cite{rinzel86}. The square-wave bursting scenario requires that the fast subsystem, when the original slow variable is considered as a parameter, is bistable between equilibria and limit cycles. It also requires that the bistability zone of the fast subsystem is bounded by two bifurcations, namely a saddle-node bifurcation of equilibria and a saddle homoclinic bifurcation so as to form a \textit{hysteretic cycle}. This notion refers to the existence of two families of attractors of the fast subsystem, each of them being bounded by a bifurcation which destabilises the associated branch. In the full system, once a suitable slow dynamics has been put on the main parameter of the fast subsystem, a hysteretic cycle allows for an alternation between a slow passage near the first branch of attractors of the fast subsystem up to the first bifurcation (in the present case, a nonsmooth fold bifurcation) and a slow passage near the second branch of attractors up to the second bifurcation (in the present case, a nonsmooth saddle homoclinic bifurcation). This alternation between two phases of slow passages, often referred to as the \textit{quiescent phase} and the \textit{burst phase}, is essential for bursting dynamics to arise, in particular square-wave bursting, like in the system we consider here. 

In the present hybrid model, we do have bistability between equilibria and limit cycles in the fast subsystem, and a hysteretic cycle. Therefore system~\eqref{eq:modelgen}--\eqref{eq:reset} can be seen as a nonsmooth hybrid square-wave burster. Note that the system~\eqref{eq:modelgen} has one fast and one slow variable, and the reset condition corresponds to a second fast (in fact, infinitely fast since the reset is instantaneous) variable. Moreover, in the simplest square-wave bursting framework, e.g. the Hindmarsh-Rose model, which is a polynomial square-wave burster, the bursting is obtained by putting a suitable slow dynamics on the parameter displaying bistability and a hysteresis cycle. This slow variable must oscillate along the hysteresis cycle of the fast subsystem, which requires a feedback term from one of the fast variables. In contrast, here, the reset instantaneously decreases the value of $v$ and increments that of $w$. In this way, due to the reset, the dynamics of $w$ is decreasing in the left half-plane $\{v<0\}$ and it is increasing each time $v$ hits the threshold value. This can be obtained in a 3D smooth slow-fast square-wave burster provided the $w$-dynamics (slow) has a term in $v$, and then the resulting dynamics would be entirely comparable to the dynamics of system~\eqref{eq:modelgen}--\eqref{eq:reset}.

Our two models with $F(v,w) = -w$ and $F(v,w) = v-b$ together with the reset rules~\eqref{eq:reset} effectively behave as nonsmooth square-wave bursting systems, as shown in Figure~\ref{fig:comp} where we compare them with a smooth square-wave burster, namely the extended Morris-Lecar model~\cite{terman91}. The main points of comparison are as follows. First, the critical manifold $S_0$ is $V$-shaped in the nonsmooth models and $S$-shaped in the smooth one, with a lower fold point that organises the transition from the quiescent phase of the bursting solution to its burst phase, and which corresponds to a saddle-node bifurcation of the fast subsystem, nonsmooth or smooth depending on the case. Second, the burst phase corresponds to a family of stable limit cycles of the fast subsystem, which in the nonsmooth system are organised via the threshold and reset lines. Finally, a homoclinic bifurcation of the fast subsystem, nonsmooth or smooth, ends the burst. It is marked by a brown star in the figure. The spike-adding transition are comparable in both cases, and mediated by canard solutions, as shown for the nonsmooth model in the present article and, for the smooth model, in e.g.~\cite{desroches2013}. 

There are nevertheless marked differences and we would like to briefly comment on three of them. Firstly, the nonsmooth framework allows for bifurcations, in particular grazing bifurcations, associated with maximal canards, unlike the smooth case where no bifurcations typically accompany this transition. Second, the nonsmooth framework allows for the existence of two different types of maximal canards, associated with two different solution branches of limit cycles with $N$ and $N+1$ resets, respectively. One providing an upper limit (for the L$_2$-norm) to the family of ``headless'' canard cycles with $N$ resets, and the other providing an upper limit to the family of canard cycles with ``head'' with $N+1$ resets (the head of the canard corresponding to the last reset, see Figure~\ref{fig:maxcan}. In the context of smooth slow-fast systems, there is only one type of maximal canards, and one can conjecture that, as the reset parameter $k$ tends to $0$, the two different maximal canards from the nonsmooth system converge towards a joint object that is akin to the smooth maximal canard. Note in passing that in both contexts the maximal canards, indeed, correspond to solutions following the repelling part of the critical manifold for the longest possible segment. However due to the different geometries of the critical manifolds in both cases, there is a geometrical difference between the location of the maximal canard in nonsmooth reset systems and that of the maximal canard in the smooth systems. The maximal canards in nonsmooth systems follows the repelling part of $S_0$ until it meets the threshold line, at which point, by definition, the orbit either has a tangent intersection with the threshold line and turns back towards the switching line (first type of maximal canard, illustrated in Figure~\ref{fig:maxcan} (c)), or a transversal intersection with the threshold line and undergoes a reset (first type of maximal canard, illustrated in Figure~\ref{fig:maxcan} (d)). In contrast, in smooth systems the maximal canard follows the repelling part of the critical manifold until its upper fold point, where the critical manifold becomes attracting again. Finally, there are differences in the bifurcation structure, with a discontinuous adding structure in the nonsmooth case (given that the number of resets must be constant on every branch), whereas the bifurcation structure is continuous (at least for small enough $\eps$) in the smooth case.

%------------------------------------------------
\section{Conclusion and perspectives}
\label{sec:conc}
%------------------------------------------------

 We have studied a slow-fast piecewise-linear adaptive integrate-and-fire model and finely investigated its bursting regime. In particular, we have analysed canard-induced spike-adding transitions, where canard segments appear upon a one-parameter variation linked with a reset. We have proven the existence of canard cycles and the transition between $N$-reset and $(N+1)$-reset cycles. The question of stability of such cycles along solution branches is quite subtle, in particular in the canard regime, and will deserve further attention in a follow-up study which will be entirely dedicated to this question.
 
  The question of persistence of these canard structure in higher-dimensions is of course extremely pertinent and we plan to work on the generalisation of the present results to systems with more fast and (more interestingly) more slow variables.
  We have considered two versions of the model, depending on the slow dynamics, and which correspond to two different excitable scenarios. Indeed, system~\eqref{eq:model1} is akin to a type-1 neuron model, even though we only considered it in its periodic regime. However, the reset-induced canard mechanism also exists in the excitable regime of~\eqref{eq:modelgen} for $F(v,w) = b-w$ (that is when $b>I$), and one could prove the existence of trajectories containing canard segments (related to the excitability threshold in this context) by using similar map-arguments as we have done here. In contrast, system~\eqref{eq:modelgen} with $F(v,w) = v-b$  is akin to a type-2 neuron model, even though here as well we have studied it only in its periodic regime. The excitable regime in both versions of the model will be an interesting topic for future work.
  
 We have performed accurate numerical analysis of the model, with both direct simulations and numerical continuation which allowed us to reliably compute families of $N$-reset limit cycles up to the canard regime. In parameter space, the spike-adding transitions are discontinuous but branches of $N$-reset and $(N+1)$-reset cycle are exponentially close to each other due to the canard regimes that bound them. In order to apply standard numerical continuation (with the software package \textsc{auto}) to this nonsmooth problem with reset, we have recasted the 2D numerical problem with periodic boundary conditions as a $2(N+2)$D equivalent two-point boundary value problem which, with suitable boundary conditions as well as stopping conditions (when the numerical solution does not correspond anymore to a solution to the original system), allowed us to compute branches of limit cycles of this system. Future work will involve improving this numerical continuation approach, in particular by implementing test functions to detect bifurcations in this setup as well as automatic stopping conditions and branch switching.
 
 Although the models we have studied are phenomenological, there exists more biophysically plausible versions of such AIF models, in particular conductance-based versions of them. In a recent work~\cite{gorski2020}, a conductance-based AIF model was extensively studied as a single unit and also within networks. Most of the dynamics uncovered in this biophysical version can be obtained with the models we studied in the present work; in particular the \textit{delayed bursting} showcased in~\cite{gorski2020} clearly correspond to the bursting cycles with canard segments that we have investigated. Both our theoretical and computational results can be adapted to the biophysical extensions, which is an interesting point of focus for future work.
 
 The link with smooth bursting models is a natural questions and we have presented elements of comparison with smooth square-wave bursting systems, since this is the type of bursting one obtains with systems~\eqref{eq:modelgen}--\eqref{eq:model2} with linear $F(v,w)$. Under parameter variation, we highlihted a number of similarities with smooth square-wave bursters in the role of canard solutions in organizing the transition to bursting as well as shaping the burst itself. To this extent, given that the subthreshold model is a PWL slow-fast system with a $V$-shaped critical manifold, there is no fold-initiated canard in this model, and the canard cycles that we computed and analysed in the present work are akin to \textit{jump-on} canards already present in many smooth slow-fast systems.\newline We also noticed important discrepancies with the smooth case, namely: the geometry of the critical manifold affecting that of the maximal canards; the presence of two types of maximal canards in the nonsmooth framework, as well as the discontinuous bifurcation structure of the spike-adding regime. An interesting future question in this direction is to obtain a deeper understanding of maximal canards in both contexts. More generally, we plan to study PWL slow-fast systems with reset as a framework for constructing and analysing neuro-inspired systems displaying complex oscillations (bursting, mixed-mode oscillations, etc.) in a minimal setup amenable to precise analysis and precise computations while retaining all salient features of their smooth counterparts. In our opinion, such systems offer the advantage of the simplified singular perturbation framework of PWL slow-fast systems~\cite{desroches2016b,desroches2018,desroches2016a} with less linearity zones due to the reset. What is more, they are relevant as models of electronic circuits and artificial neurons (both PWL systems and systems with resets) and hence have the potential to provide a very good mathematical template to study neuronal and neuromorphic systems.

\section*{Acknowledgements}
SR was supported by Ikerbasque (The Basque Foundation for Science), by the Basque Government through the BERC 2018-2021 program and by the Ministry of Science, Innovation and Universities: BCAM Severo Ochoa accreditation SEV-2017- 0718 and through project RTI2018-093860B-C21 funded by (AEI/FEDER, UE) with acronym ``MathNEURO''.

\end{document}